\documentclass[reqno,11pt,twoside]{article}
\usepackage{amsmath,amsfonts,amsfonts,amssymb}
\usepackage{amsbsy}
\usepackage{amsthm}
\usepackage{mathrsfs}
\usepackage{graphicx}
\usepackage{amscd}

\theoremstyle{break}
\newtheorem{lemma}{Lemma}

\newtheorem{proposition}[lemma]{Proposition}
\newtheorem{theorem}[lemma]{Theorem}
\newtheorem{corollary}[lemma]{Corollary}
\newtheorem{conjecture}[lemma]{Conjecture}

\newtheorem{remark}[lemma]{Remark}
\newtheorem{definition}[lemma]{Definition}

\newcommand \QQ {{\mathbb Q}}
\newcommand \PP {{\mathbb P}}

\newcommand{\CC}{\ensuremath{\mathbb{C}}}
\newcommand \FF {{\mathbb F}}
\newcommand \GG {{\mathbb G}}
\newcommand \AAA {{\mathbb A}}
\newcommand \RR {{\mathbb R}}
\newcommand \ZZ {{\mathbb Z}}
\newcommand \LL {{\mathbb L}}
\newcommand \NN {{\mathbb N}}
\newcommand \cQ {{\mathcal Q}}

\newcommand \cI {{\mathcal I}}

\newcommand \cV {{\mathcal V}}

\newcommand \cP {{\mathcal P}}

\newcommand{\q}{/\!\!/}

\newcommand{\rank}{\mathrm{rank\,}}
\newcommand{\corank}{\mathrm{corank\,}}

\newcommand{\Sing}{\mathrm{Sing}}
\newcommand{\Spec}{\mathrm{Spec}\,}

\newcommand{\Mat}{\mathrm{Mat}}

\newcommand\myphi{\scalebox{1.3}{$\varphi$}}

\newcommand{\1}{{\rm 1\hspace*{-0.4ex}%
\rule{0.1ex}{1.52ex}\hspace*{0.2ex}}}

\begin{document}

\begin{center}
\huge\textsf{The $c_2$ invariant is invariant}\\
\medskip
\normalsize \textsc{ Dmitry Doryn}\\
dmitry@ibs.re.kr

\end{center}
\begin{abstract}
\noindent The $c_2$ invariants in all 4 different representations of the Feynman period (parametric and dual parametric representations, position and momentum spaces) coincide for all log-divergent graphs that satisfy the combinatorial condition called duality admissibility. We check this condition for a good subspace of graphs, for instance for all planar graphs. After the result in \cite{D3}, the coincidence holds for all physically relevant graphs. 
\end{abstract}

\section*{Introduction}
A good progress was done in the evaluation of the Feynman integrals in QFT in the last decades, especially in $\phi^4$ theory. Nevertheless, this is still a big problem for graphs starting with 9 loops. An interesting algebraic direction of research is a relation between the Feynman period and the number of rational points of the poles of the Feynman differential form over finite fields.

In this article we continue and extend the work started by F. Brown, O. Schnetz and K. Yeats in \cite{BSY} and prove that the part of a point-counting function is the same for all 4 different representations of the Feynman period.

For a connected graph $G$ with $N_G$ edges, $n_G+1$ vertices, and $h_G:=N_G-n_G$ cycles, the graph polynomial and the dual graph polynomial are defined by
\begin{equation}\label{1}
\Psi_G=\sum_{T} \prod_{e\notin T} \alpha_e\quad \text{and} \quad\myphi_G= \sum_{T} \prod_{e\in T} \alpha_e \in \ZZ[\alpha_1,\ldots,\alpha_{N_G}],
\end{equation}
with $\alpha_i$s the Schwinger parameters (edge variables) and $T$ running over all spanning trees of $G$.
Recall that a graph $G$ is said to be log-divergent if $N_G=2h_G$, and a log-divergent graph $G$ is primitive log divergent if for any proper subgraph $\gamma\subset G$ the following inequality holds: $2h_{\gamma}<N_{\gamma}$. 
It the case $G$ is log-divergent, one has the associated Feynman period $I_G$ defined by an integral of a differential form with double poles along $\Psi_G=0$. Similarly, the other form of the Feynan period is the integral $I^{dual}_G$ with poles along $\myphi_G$ with inverted variables. The more natural representation for physicists  is the one in momentum space ($I^{mom}_G$)(see \cite{IZ}), while the position space ($I^{pos}_G$) is where some good techniques effectively help in the computations, as Gegenbauer polynomials (\cite{CKT}), etc. The connection of these different approaches are shown on the following diagram:\medskip\\
\mathstrut\hspace{6em}$
\begin{array}[c]{ccc}
\begin{tabular}{c}
\large\texttt{ parametric} \\
\large\texttt{space }
\end{tabular}
  & \begin{tabular}{c}
\footnotesize{Schwinger}\\[-4pt] $\longleftrightarrow$\\[-5pt]
\footnotesize\text{trick }
\end{tabular}  & 
  \begin{tabular}{c}
\large\texttt{momentum} \\
\large\texttt{space }
\end{tabular}\\
\rotatebox{90}{\begin{tabular}{c}
\footnotesize{Cremona}\\[-5pt] $\longleftrightarrow$\\[-5pt] \footnotesize{transform.}
\end{tabular}} & & \rotatebox{90}{\begin{tabular}{c}
\footnotesize{Furier}\\[-5pt] $\longleftrightarrow$\\[-5pt] \footnotesize{transform.}
\end{tabular}}\\
 \begin{tabular}{c}
\large\texttt{dual}\\
\large\texttt{ parametric} \\
\large\texttt{space }
\end{tabular} & 
\begin{tabular}{c}
\footnotesize{Schwinger}\\[-4pt] $\longleftrightarrow$\\[-5pt]
\footnotesize\text{trick }
\end{tabular}
&
\begin{tabular}{c}
\large\texttt{position} \\
\large\texttt{space }
\end{tabular}
\end{array}
$
\begin{center}
\textsc{Figure 1}
\end{center}
For a primitive log-divergent graph the 4 integrals defined in this spaces give the same value (up to multiplication by $\pi^i$). See \cite{Sch}, Section 2 for more explanation. 

In practice, it's quite complicated  to compute the period $I_G$ analytically in any of these representations, and usually can be done only for small graphs. On the other hand, the values of $I_G$ for many known examples of graphs are lying in the $\QQ$-algebra spanned by multiple zeta values (MZV), see \cite{BrdKr}, \cite{Sch}. One knows the deep connection of MZV to algebraic geometry and to mixed Tate motives.  This motivates the study of the arithmetic and algebraic nature of the poles of $I_G$, i.e. of the graph hypersurface $X_G$ defined by the vanishing of $\Psi_G=0$ in affine (or projective) setting. 

For the structure of $\Psi_G$ see \cite{Br}, \cite{BrY}. The Kontsevich conjecture on the number of rational points of $X_G$ was discussed in \cite{BB}, \cite{Sch}, \cite{D2}, \cite{BrSch}. The cohomological  approach for study of $X_G$ and motivic point of view on the Feynman period can be found in \cite{BEK}, \cite{D}, \cite{BrD}.

Recall that for $G$ with $n_G\geq 2$ one has the congruence $\#X_G(\FF_q)\equiv 0 \mod q^2$ counting $\FF_q$-rational points for a fixed $q$ of (the base change to $\FF_q$ of) $X_G$. One defines
\begin{equation}
c_2(G)_q:=\#X_G(\FF_q)/q^2 \mod q.
\end{equation}
Motivated by the known examples, one makes the following conjecture (see Conjecture 5 in \cite{BrSch}):
\begin{conjecture}
If $I_{G_1}=I_{G_2}$ for two primitive log-divergent graphs $G_1$ and $G_2$, then $c_2(G_1)_q=c_2(G_2)_q$.
\end{conjecture}
\noindent In other words, $c_2$ invariant should play a role of a discrete analogue of the Feynman period.
One can even define the $c_2(G)$ invariant in the Grothendieck ring $K_0(Var_k)$ of varieties over a field, and can ask for the same question (this is partially done in \cite{BSY} and in our article in dual setting). Since we have no Chevalley-Warning vanishing in $K_0(Var_k)$  (by the result of Huh in \cite{Hu}), and since the Grothendieck ring has not only zero-divisors but also $\LL$ is a zero divisor (see \cite{Bo}), the question becomes more complicated.

It was natural to expect the existence and coincidence of the analogues of the $c_2(G)_q$ invariants in all 4 spaces in Figure 1, since the values of the integral representations coincide.  

The relation on the level of the $c_2$ invariant in the upper row in Figure 1 was studied in \cite{BSY}. There was defined the $c_2^{mom}(G)_q$ invariant for a graph with $N_G\leq 2h_G$, $h_G\geq 2$ in Proposition-Definition 17 in \cite{BSY}, and then there was proved the following theorem (see Theorem 18 loc. cit.):
\begin{theorem}
Let $G$ be a log-divergent graph (i.e. $N_G=2h_G$) with $h_G\geq 3$. Then the $c_2$ invariants in parametric and momentum spaces coincide:
\begin{equation}
c_2^{mom}(G)_q=c_2(G)_q.
\end{equation}
\end{theorem}

In this article we discuss the analogues of $c_2$ invariant for the remaining two spaces : dual parametric and position spaces. 

In section 1 we study the properties of the dual graph polynomial $\myphi_G$ and define $c_2(G)^{dual}_q$. The situation is very similar (but dual) to the case of $\Psi_G$.

Section 2  contains the computation of the classes of the dual graph hypersurface and of its singular locus in the Grothendieck ring, this is a translation of the results for $\Psi_G$ from \cite{BSY} to our setting with minor modifications.

In section 3 we do the computations for point counting functions in position space. We try to follow a similar strategy to the one was used in \cite{BSY} for the case of momentum space. We define $c_2^{pos}(G)_q$ out of the configuration of quadrics (in the vertex variables) in the denominator of the differential form of $I_G^{pos}$, and then prove
\begin{theorem}
For a log-divergent graph $G$ with $n_G\geq 3$, the $c_2$ invariants in the dual parametric space and in position space coincide:
\begin{equation}
c_2^{dual}(G)_q=c_2^{pos}(G)_q.
\end{equation}
\end{theorem}
\noindent After Theorems 2 and 3, the remaining part for the coincidence of the $c_2$ invariants in all 4 representations is to prove that $c_2$ respects the Cremona transformation in the left column of Figure 1. This is the content of Section 4. For the proof we need to restrict to the graphs we call duality admissible (see Definition \ref{def26}). This class contains log-divergent graphs which are planar or have enough triangles. This additional conditions come from the fact that any log-divergent graph always has a vertex of degree $\leq 3$ but not always has a cycle of length $\leq 3$. We make a conjecture that the conditions are always satisfied.
\begin{conjecture}
Let $G$ be a log-divergent graph with $h_G,n_G\geq 3$. Then $G$ is duality admissible.
\end{conjecture}
\noindent The main theorem of the article is the following (Theorem \ref{Thm2})
\begin{theorem}
Let $G$ be a duality admissible graph. Then the $c_2$ invariants for parametric and for dual parametric representations coincide:
\begin{equation}
c_2^{dual}(G)_q=c_2(G)_q.
\end{equation} 
\end{theorem}
\noindent This part (left column) of the Figure 1 was assumed to be the hardest one, see the discussion at the end of Section 3 in \cite{BSY}. Putting everything together, we finally get
\begin{theorem}
For any duality admissible graph $G$ with $h_G,n_G\geq 3$, the $c_2$ invariants in all four spaces on Figure 1 coincide. 
\end{theorem}
There are infinite series of graphs, like $WS_n$ and $ZZ_n$, for which one can compute the Feynman period $I_G$, all these series consist of planar graphs. Our methods here cover these graphs, since we have proved that all planar graphs are duality admissible, see Corollary \ref{cor34}. Several good interesting graphs are also planar. For example, one of the known counter-examples to Kontsevich conjecture is planar, see Section 6.3 in \cite{BrSch}. 

In \cite{D3}, we have found a new approach for proving the duality-admissibility called "a 4-face formula" that works for not necessarily planar graphs, possibly without triangles. This allows us to prove Conjecture 4 for every graph $G$ such that each it's subquotient graph has a loop of length at most 4. That is enough for all physically relevant graphs. By this we mean that the minimal graph that we cannot cover has 18 loops, it is  outside the known special infinite series and its period is very far away from being calculated in any sense. 
\medskip\\
\textbf{Acknowledgements} 

\noindent The author is very thankful to Dirk Kreimer and Alexander von Humboldt foundation for financial support. The preparation of the final version was supported by the Max Plank Institute f\"ur Mathematik, Bonn.
\section{Dual graph polynomials}

We start with a graph $G$ that consists of the set of vertices $V(G)$ and the set of edges $E(G)$. We define $N=N_G:=|E(G)|$ and $n_G:
=|V(G)|-1$. The Euler formula then implies that $h_G:=N_G-n_G$ is the \emph{loop number} (number of "independent" cycles). This $h_G$ can be also seen as the rank of the first homology group of $G$ (\cite{BEK}, Section 2). We use the index set $I_{N}:=\{1,\ldots,N_G\}$ for labelling of the elements of the set $E(G)$ , so $E(G):=\{e_i\}_{i\in I_{N}}$. To each edge $e_i$ we associate a variable (Schwinger parameter) $\alpha_i$.

For a connected graph $G$, one defines the first Symanzik polynomial, or simply the graph polynomial, denoted by $\Psi_G$ as in (\ref{1}). Equivalently, $\Psi_G$ can be defined as the determinant of the matrix 
\begin{equation}\label{d2}
M(G) =
\left(
\begin{array}{c|c}
\Delta(\alpha)& E\\
\hline
-E^{T^{\mathstrut}} & 0 
\end{array}
\right) \in \Mat_{N+n,N+n}(\ZZ[\{\alpha_i\}_{i\in I_N}]),
\end{equation}
where $\Delta(\alpha)$ is the diagonal matrix with entries $\alpha_1,\ldots,\alpha_N$, and $E\in \Mat_{N,n}(\ZZ)$ is the incidence matrix after deleting the last column, $N=N_G$, $n=n_G$ (see \cite{Br}, Section 2.2). Out of this matrix, one can define the Dodson polynomials $\Psi^{I,J}_{G,K}$ by $\Psi^{I,J}_{G,K}:= \det M(G)(I;J)_K$, where $M(G)(I;J)_K$ obtained from $M(G)$ after removing rows indexed by $I$ and columns indexed by $J$, and after putting $\alpha_t=0$ for all $t\in K$. For simplicity, we usually write $\Psi^I_{K}$ for $\Psi^{I,I}_{G,K}$. These Dodson polynomials satisfy many identities like contraction-deletion formula, the first and second Dodson identities, etc. (see \cite{Br}).

In contrast to $\Psi_G$, one also has 
\begin{equation}
\myphi_G:= \sum_{T} \prod_{e\in T} \alpha_e \in \ZZ[\{\alpha_i\}_{i\in I_N}],
\end{equation}
the \emph{dual graph polynomial}. To explain the relation between  the graph polynomials and the dual one's, we define the \emph{Cremona transformation} $\iota:\ZZ[\{\alpha_i\}_{i\in I_N}]\rightarrow \ZZ[\{\alpha_i\}_{i\in I_N}]$ as follows: for a polynomial $P\in\ZZ[\{\alpha_i\}_{i\in I_N}]$ dependent on the variables indexed by $I$,  $\iota(P)(\alpha_1,\ldots,\alpha_N)= P(\frac{1}{\alpha_1},\ldots,\frac{1}{\alpha_N})\prod_{i\in I} \alpha_i$. We often call the application of this transformation simply the dualization.
By the very definition, $\myphi_G=\iota(\Psi_G)$. Define $\myphi^I_J:=\iota(\Psi^J_I)$.
Starting with the contraction-deletion formula for a graph polynomial, $\Psi_G=\Psi^k_G\alpha_k+\Psi_{G,k}$ (Formula (11) in \cite{BSY}), inverting the variables and multiplying with $\prod_{i\in I_N\backslash k} \alpha_i $, one gets the similar-looking \emph{contraction-deletion formula} for the dual graph polynomial:
\begin{equation}\label{c10}
\myphi_G=\myphi^k_G\alpha_k+\myphi_{G,k}
\end{equation}
for any $k\in I_N$. Moreover, 
$\myphi^k_G=\myphi_{G\q e_k}$ and $\myphi_{G,k}=\myphi_{G\backslash e_k}$ 
with $G\backslash e_k$ (resp. $G\q e_k$) 
denoting the graph $G$ after deletion (resp. contraction) of the edge $e_k$. 

We can easily derive the formulas for special cases of $G$:
\begin{enumerate}
\item[1).] If an edge $e_1\in E(G)$ forms a tadpole (self-loop), then 
\begin{equation}\label{e100}
\myphi_G=\myphi_{G\backslash 1}.
\end{equation}
\item[2).] If two edges $e_1,e_2\in E(G)$ form a cycle of length 2 (double-edge), then
\begin{equation}\label{e101}
\myphi_G=\myphi_{G\backslash 1\q 2}(\alpha_1+\alpha_2)+\myphi_{G\backslash 12} .
\end{equation}
\end{enumerate} 
For $I\cap J =\emptyset$ and $|I|=|J|$ define $\myphi_G^{I,J}:=\iota(\Psi_G^{I,J})$. Sometimes we fix $G$ and omit the subscript to make the formulas more readable.
Fix two indexes $i\neq j$ and consider the special case of the (first) Dodgson identity for $\Psi_G$ (see \cite{Br}, (20)):
\begin{equation}
\Psi^i\Psi^j+\Psi\Psi^{ij}=(\Psi^{i,j})^2.
\end{equation}
This identity follows from the (studied by Dodgson) identities on the minors of a symmetric matrix, knowing that $\Psi_G$ is a determinant of the matrix (\ref{d2}) that can be made symmetric after possible inversion of the signs in the last rows. 
Dualizing the equation above, we get 
\begin{equation}
\myphi_i\myphi_j+\myphi\myphi_{ij}=(\myphi^{i,j})^2\alpha_i\alpha_j.
\end{equation}
Applying (\ref{c10}) twice and taking the coefficients of $\alpha_i\alpha_j$, we obtain
\begin{equation}\label{c14}
\myphi_i^j\myphi_j^i+\myphi^{ij}\myphi_{ij}=(\myphi^{i,j})^2.
\end{equation}  
Using the expansions of $\myphi$, $\myphi^i$ and $\myphi^j$ in $\alpha_i$ and $\alpha_j$ (by (\ref{c10})), one computes
\begin{equation}\label{c15}
\myphi^j\myphi^i+\myphi^{ij}\myphi= \myphi_i^j\myphi_j^i+\myphi^{ij}\myphi_{ij}=(\myphi^{i,j})^2.
\end{equation}
\noindent More generally, define the dual Dodgson polynomials by
\begin{equation}\label{c16}
\myphi^{IS,JS}_{G,K}:=\iota(\Psi^{IK,JK}_{G,S})
\end{equation}  for any $I,J,K,S\subset I_N$ pairwise non-overlapping with $|I|=|J|$. One immediately gets $\myphi^{IS,JS}_{G,K}:=\iota(\Psi^{I,J}_{G\backslash K\q S})=\myphi^{I,J}_{G\backslash K\q S}$.

With this definition we get the non-natural $\myphi^{I,I}=\myphi_I$ from the point of view of the graph polynomial, but the identities on dual Dodgson polynomials become looking very similar to the case of $\Psi_G$. The dual Dodgson polynomials satisfy 
\begin{equation}
\myphi^{I,J}_{G,K}=\pm\myphi^{It,Jt}_{G,K}\alpha_t\pm\myphi^{I,J}_{G,Kt}
\end{equation}
for $t\in I_N\backslash (I\cup J\cup K)$ and possibly overlapping $I,J$. The signs in the formula can be explained by using spanning forest polynomials similar to the case of $\Psi_G$, see \cite{BrY}, Section 2. 
Recall (\cite{BrSch}, Section 2.2) that the first Dodgson identity is
\begin{equation}
\Psi^{ISx,JSx}_{K}\Psi^{ISa,JSb}_{Kx}- \Psi^{IS,JS}_{Kx}\Psi^{ISax,JSbx}_K= \pm\Psi^{ISx,JSb}_{K}\Psi^{ISa,JSx}_{K}
\end{equation}  
for $I,J,S,K\subset I_N$ and non-overlapping, $|I|=|J|$ and $a,b,x\in I_N\backslash(I\cup J\cup S\cup K)$. The sign depends on the order of $a$,$b$ and $x$. Dualizing this, we get the \emph{(dual) Dodgson identity} (of the first type) for the dual Dodgson polynomials
\begin{equation}\label{c18}
\myphi^{IKx,JKx}_{S}\myphi^{IKa,JKb}_{Sx}- \myphi^{IK,JK}_{Sx}\myphi^{IKax,JKbx}_{S}= \mp\myphi^{IKx,JKb}_{S}\myphi^{IKa,JKx}_{S}.
\end{equation} 
We can also derive the Dodgson identity of the second type for dual Dodgson polynomials by dualizing the one for Dodgson polynomials :
\begin{equation}\label{c20}
\myphi^{IKax,JKx}_S\myphi^{IKb,JK}_{Sx}-
\myphi^{IKa,JK}_{Sx}\myphi^{IKbx,JKx}_S =\pm\myphi^{IKx,JK}_S\myphi^{IKab,JKx}_{S.},
\end{equation}
where $I,J,S,K\subset I_N$ are non-overlapping, $|J|=|I|+1$ and $a,b,x\in I_N\backslash(I\cup J\cup S\cup K)$.

Define the resultant $[f,g]_k$ of two polynomials $f=f^k\alpha_k+f_k$ and $g=g^k\alpha_k+g_k$ linear in a variable $\alpha_k$ by $[f,g]_k=f^kg_k-f_kg^k$.
The next lemma is the analogue of Lemma 21 in \cite{BSY}.
\begin{lemma}
For any 3 distinct edges indexed by $i,j,k$ of $G$ the following identity holds
\begin{equation}
[\myphi^i,\myphi^j]_k=
\myphi^{ij,jk}\myphi^{j,k}
-\myphi^{ij,jk}\myphi^{i,k}.
\end{equation}
\end{lemma}
\begin{proof}
The proof is obtained from that proof in \cite{BSY} by replacing $\Psi_G$ with $\myphi_G$ because of the similarity of the Dodgson identities (\ref{c18}) and the contraction-deletion formulas (\ref{c10}) for $\Psi_G$ and $\myphi_G$.
\end{proof}
\begin{corollary}\label{cor7}
Fix an element $k\in I_N$ and let $\cI$ be the ideal of $\QQ[\{\alpha_i\}_{i\in I_N}]$ generated by $\myphi^k$ and $\myphi_k$. Then
\begin{equation}
[\myphi^i,\myphi^j]_k\in Rad (\cI)
\end{equation}
\begin{proof}
Using (\ref{c18}) and the linearity of the resultant, one computes
\begin{equation}
(\myphi^{i,k})^2=[\myphi^i,\myphi_i]_k=[\myphi,\myphi^i]_k=\myphi^k\myphi^i_k-\myphi_k\myphi^{ik}\in\cI .
\end{equation}
Thus $\myphi^{i,k}\in\cI$ and similarly $\myphi^{i,j}\in\cI$. The lemma above implies the statement. 
\end{proof}
\end{corollary}
\begin{proposition}\label{p6} Let $G$ be a graph with edges $E(G)$ labelled with the set $I_N$ and let $I=\{1,2,\ldots,t\}\subset I_N$ be a subset.
\begin{enumerate} 
\item[\textbf{i).}] If the edges labelled with $I$ form a corolla (all the edges incident to one fixed vertex), then
\begin{equation}\label{c100}
\myphi_{G,1}=\sum_{i\in I\backslash 1} \lambda_i \alpha_i \myphi^{1,i}_G,\text{\;\;where\;} \lambda_i=\pm 1.
\end{equation}
\item[\textbf{ii).}] If the edges labelled with $I$ form a cycle (topological loop), then  
\begin{equation}\label{c101}
\myphi_G^1=\sum_{i\in I\backslash 1} \lambda_i \myphi^{1,i}_G,\text{\;\;where\;} \lambda_i=\pm 1.
\end{equation}
\end{enumerate}
\end{proposition}
\begin{proof}
For part \textbf{(i)}, we start with the formula for the graph polynomial with given edges forming a corolla $\Psi^1_G=\sum_{i\in I} \lambda_i \Psi^{1,i}_G$
(see Lemma 31 in \cite{Br}). Dualization immediately gives (\ref{c100}). For the part \textbf{(ii)}, with edges forming a cycle, we can just dualize the formula $\Psi_{G,1}=\sum_{i\in I} \lambda_i \alpha_i \Psi^{1,i}_G$ that was proved in \cite{BSY}, Proposition 24.
\end{proof}
\begin{corollary} \label{cor9}
Let $G$ be a connected graph with more than 1 edge and let us fix any edge of $G$, say $e_1$. Then there exists a subset $I=\{1,\ldots,t\}\subset I_N$ such that
$\myphi_1$ lies in the radical $Rad(\cI)$ of the ideal $\cI\subset \ZZ[\{\alpha_i\}_{i\in I_N\backslash 1}]$ spanned by $\myphi^1$, and $\myphi^{1i}_G$ for all $i\in I\backslash 1$.
\end{corollary}
\begin{proof}
Since $G$ is connected, one of the endpoints of the edge $e_1$ has degree bigger than 1, define this endpoint by $v$ and its degree by $d_v$. Let $I=\{1,\ldots,d_v\}\subset I_N$ be the set that labels the edges of the corolla of $v$. Using the Dodgson identity (\ref{c15}), one computes
\begin{equation}
(\myphi^{1,i})^2=[\myphi_i,\myphi^i]_1=[\myphi_i+\myphi^i\alpha_i,\myphi^i]_1=\myphi^1\myphi^i_1-\myphi_1\myphi^{1i}\in \cI.
\end{equation}
Thus, $\myphi^{1,i}\in Rad(\cI)$ for each $i\in I\backslash 1$.  Now Proposition \ref{p6}, part (\textbf{i}) implies the statement.
\end{proof}

\noindent We return to the representation for $\Psi_G$ as a determinant of the matrix (\ref{d2}).
Working with blocks, we can modify the matrix as follows:
\begin{equation}
\left(
\begin{array}{c|c}
\Delta(\alpha)& E\\
\hline
-E^{T^{\mathstrut}} & 0 
\end{array}
\right) 
\left(
\begin{array}{c|c}
J_N & -\Delta(\frac{1}{\alpha})E\\
\hline
0 & J_n 
\end{array}
\right) = 
\left(
\begin{array}{c|c}
\Delta(\alpha)& 0\\
\hline
-E^{T^{\mathstrut}} & E^{T}\Delta(\frac{1}{\alpha})E 
\end{array}
\right) 
\end{equation}
with $\Delta(\frac{1}{\alpha}):=\Delta(\alpha)^{-1}$. Here $J_d$ denotes the $d\times d$ identity matrix for $d= n, N$. Taking determinants of both sides, we get 
\begin{equation}\label{d9}
\Psi(\alpha)\cdot 1=\prod_{i\in I_N} \alpha_i \cdot \det 
(E^{T}\Delta(\frac{1}{\alpha})E).
\end{equation} 
Substituting $\alpha_i\mapsto 1/\alpha_i$, $i\in I_N$, one obtains
\begin{equation}\label{d18}
\myphi(\alpha)=\prod_{i\in I_N}\alpha_i\cdot\Psi(1/\alpha)=\det (E^{T}\Delta(\alpha)E).
\end{equation}
We define  
\begin{equation}\label{d21}
P_G(\alpha):=E^T\Delta(\alpha)E \;\in\; \Mat_{n\times n}(\ZZ[\{\alpha_i\}_{i\in I_N}]),
\end{equation}
then $\myphi_G=\det P_G(\alpha)$ as above.
One easily sees that the matrix $P(\alpha)$ can be written as $P(\alpha)=\sum \alpha_i P_i$, where $P_i\in\Mat_{n, n}(\ZZ)$ for an edge $e_i=(v_s,v_t)$ has entry 1 at $(s,s)$ and $(t,t)$, $-1$ at $(s,t)$ $(t,s)$, and 0 elsewhere (the special case is when one of the endpoints of the edge is the last variable that corresponds to the removed column of $E$, then the matrix has only one entry). 

Now we are going to diagonalize $P(\alpha)=P_G(\alpha)$ with respect to the certain $n$ variables (modulo the others).

\begin{proposition}\label{lemma_diag}
Let $G$ be a connected log-divergent graph and let $T$ be a spanning tree of $G$. Then there exists a matrix $\widetilde{P}(\alpha)\in \Mat_{n\times n}(\ZZ[\alpha])$ obtained from $P(\alpha)=P_G(\alpha)$ by elementary row and column operations such that for any $i$, $1\leq i\leq n_G$, there exists a variable appearing at the only entry $(\widetilde{P})_{i,i}$. 
\end{proposition}
\begin{proof}
Assume for a moment that we have a Hamiltonian path in our graph, that is a connected subgraph $T$ with consecutive edges $e_1,\ldots, e_n$ with no loops or branch points which contains all the vertices. Then, permuting the edges we can write the matrix $P$ in the form 
\begin{equation}
\left(
\begin{array}{cccc}
 \ddots & \vdots & \vdots & \vdots\\
\ldots & \alpha_{n-1}+\alpha_{n-2}+c& -\alpha_{n-1} &  c_{n-2\,n} \\
\ldots& -\alpha_{n-1} & \alpha_n+\alpha_{n-1}+b & -\alpha_n\\
\ldots &  c_{n\, n-2} & -\alpha_n &\alpha_n+a 
\end{array}
\right)
\end{equation}
where $a,b,c$ and $c_{ij}$ do depend only on the other variables $\alpha_e$ for $e\in E(G)\backslash E(T)$ (each non-specified entry $(i,j)$ of the matrix is denoted by $c_{ij}$). The variable $\alpha_n$ appears only in the 4 shown entries in the matrix. In general, $\alpha_i$ for $e_i\in E(T)$ is contained in 4 entries $(P)_{t,s}$, $i-1\leq t,s\leq i$. Consider the following operation $op(i,j)$: add the $i$-th row to the $j$-th row, and then add the $i$-th colomn to the $j$-th one. Apply $op(n,n-1)$ to the matrix above. Then the matrix takes the form
\begin{equation}
\left(
\begin{array}{cccc}
 \ddots & \vdots & \vdots & \vdots\\
\ldots & \alpha_{n-1}+\alpha_{n-2}+c& -\alpha_{n-1} &  c'_{n-2\,n} \\
\ldots& -\alpha_{n-1} & a+\alpha_{n-1}+b & a\\
\ldots &  c'_{n\, n-2} & a &\alpha_n+a 
\end{array}
\right).
\end{equation}
The variable $\alpha_n$ sits only at bottom right corner. Similarly, doing the basic operations step by step $op(n-1,n-2), op(n-2,n-3), \ldots op(2,1)$, we can bring all the variables $\alpha_e$, $e\in E(T)$ to the diagonal, i.e. $\alpha_e$ appears only at the entry $(e,e)$, as desired.

Unfortunately, the statement about the existence of the Hamiltonian path similar to that one of the Hamiltonian cycle is proved only for graphs with big enough degrees of the vertices and seems to be wrong for primitive log-divergent graphs with a big loop number. We try to modify the proof above. 

Consider now $T$ a given spanning tree of $G$ with edges $E(T)=E'\subset E(G)$, $|E'|=n$ and prove that the matrix $P$ can be transformed into the matrix where the variables $\alpha_e$, $e\in E'$, appear on and only on the diagonal. To do this we can forget the other variables (put the variables $\alpha_e$ equal zero for $e\in E''=E(G)\backslash E'$). 
Let's (re)number the edges of $T$ in the following way. Take a vertex which is not a branch point of $T$ to be a (top) root of the tree, fixing some planar embedding, and number to edges going from top to bottom and from left to right. More precisely, if we come to the branch point then we start with the left branch. When the left branch  is numbered (had come do a leaf), we return to the last branch point and go on with the next (from left to right) branch. 

One can get the intuition of the numeration algorithm by analysing the following example of a spanning tree $T$ of a graph with 7 vertices.\\
\begin{picture}(120,100)(-150,0)
\put(10,10){\line(2,3){32}}
\put(42,58){\line(2,-3){32}}
\put(26,34){\line(2,-3){16}}
\put(42,58){\line(0,1){27}}
\put(42,85){\circle*{4}} \put(42,58){\circle*{4}}
\put(10,10){\circle*{4}} \put(74,10){\circle*{4}} \put(26,34){\circle*{4}} \put(42,10){\circle*{4}} \put(58,34){\circle*{4}}
\put(32,68){$1$} \put(24,46){$2$} 
\put(8,20){$3$} \put(38,20){$4$}\put(54,46){$5$} \put(70,20){$6$} 
\put(18,-15){\textsc{Figure 2}}
\end{picture}
\bigskip\\

\noindent The root will be the vertex we throw away in the procedure of construction of the block $E$ in the matrix for $\Psi_G$.
The matrix $P(\alpha)$ for this example modulo the ideal $\cI_T\subset \ZZ[\{\alpha_e\}_{e\in E(G)}]$ generated by $\alpha_e, e\in E''$, takes the form:
$$
\left(
\begin{array}{cccccc}
 \alpha_1+\alpha_2+\alpha_5 & -\alpha_2 & \circ & \circ & -\alpha_5& \circ\\
-\alpha_2 & \alpha_2+\alpha_3+\alpha_4 & -\alpha_3 & -\alpha_4 & \circ& \circ\\
 \circ & -\alpha_3 & \alpha_3 & \circ & \circ& \circ\\
 \circ & -\alpha_4 & \circ & \alpha_4 & \circ& \circ\\
-\alpha_5 & \circ & \circ & \circ & \alpha_5+\alpha_6& -\alpha_6\\
\circ & \circ & \circ & \circ& -\alpha_6 & \alpha_6\\
\end{array}
\right)
$$
Here $\circ$ denotes an entry congruent to 0 modulo $\cI_T$. Doing the basic operations $op(i,j)$ for the pairs of rows and columns $(i,j)$ equal $(4,2), (3,2), (6,5)$, $(5,1)$ and $(2,1)$ consequently, one gets the diagonal matrix with entries $\alpha_1,\ldots,\alpha_6$.

Now consider the case of a general connected log-divergent graph $G$ with a spanning tree $T$. We diagonalize the matrix $P(\alpha)$ by induction on the number $m$ of branch points of $T$. For $m=0$ this is the case of a Hamiltonian path described above. Assume that for smaller $m$ and for all graphs the desired matrix is build. Consider the branch point $R$ with the biggest depth in the rooted tree (the lowest on the picture similar to the example above) or the leftmost one of such points (if several). According to the numeration of edges, the left branch consists of the edges $e_{s},\ldots, e_{s+p}$ for some $s,p\geq 1$. Since the leftmost branch of $R$ has no more branch points, we can diagonalize this block as in the case $m=0$ by applying $p$ basic operations. This corresponds to $op(3,2)$ in the example. The variables $\alpha_{s},\ldots,\alpha_{s+p}$ are brought to the diagonal. After forgetting these p  rows and columns with indeces from $s+1$ to $s+p$, the diagonalization of the remaining part follows from the induction hypothesis for the tree $T\q \{e_{s+1}\ldots e_{s+p}\}$. The matrix $\widetilde{P}(\alpha)$ is constructed.
\end{proof}

\begin{definition} Let $k$ be a field, $char(k)=0$. The dual graph hypersurface $Z_G$ of a connected graph $G$ is defined by the vanishing of $\myphi_G$:\; 
$Z_G:=\cV(\myphi_G)\subset \AAA^{N_G}_{k}$.
\end{definition}
\begin{definition}\label{def7}
Define the singular locus of the dual graph hypersurface $Z_G$ by
\begin{equation}
\Sing(Z_G):=\big\{\alpha\in\AAA^{N_G}_k\Big|\myphi_G(\alpha)=\frac{\partial}{\partial{\alpha_i}}\myphi_G(\alpha)=0,\; \forall i\leq N_G\big\}.
\end{equation}
\end{definition}

\begin{proposition}\label{p13}
Assume that the first $n_G$ edges of $G$ form a spanning tree. Then the ideal of \;$\Sing(Z_G)$ in $k[\{\alpha\}_{i\in I_N}]$ is
\begin{equation}\label{d30}
\cI(\Sing(Z_G))=k\big[\{\alpha\}_{i\in I_N}\big]\left\langle \myphi_G, \frac{\partial}{\partial{\alpha_i}}\myphi_G \;\Big|\; i\leq n_G\right\rangle,
\end{equation}
(is generated by the derivatives for the only $n_G$ edges).
\end{proposition}
\begin{proof}
The inclusion of the right hand side of (\ref{d30}) into the left one is clear. So we are going to prove the opposite inclusion, that is: $\myphi^i_G\in \cI'$ for all $i\in I_{N}$, where $\cI':=\langle\myphi_G, \myphi^i_G|\; i\leq n_G\rangle$.  Denote by $T$ the tree formed by the edges $e_1,\ldots, e_{n_G}$.
Recall that in Propostion \ref{lemma_diag} we have constructed the matrix $\widetilde P(\alpha)$ that is a "diagonalization" of $P(\alpha)$ with respect to $n_G$ variables corresponding to the edges of a given spanning tree $T$. We denote $\widetilde P(\alpha)$ by $P(\alpha)$ again. 
After renumbering of the variables we can assume that $\alpha_i$ is only in (a linear summand of) $P^{i,i}$ for $i=1,\ldots,n_G$. Here $P^{I,J}=P^{I,J}(t)$, $I,J\subset I_N$ denotes the matrix that we get from $P(\alpha)$ after deleting $I$ rows and $J$ columns. Thus  $\myphi^i_G(\alpha)=P^{i,i}(\alpha)$ for any $i=1,\ldots,n_G$.
Consider any edge $e_j$, $j>n_G$ with endpoints $v_s$ and $v_t$. Since $T$ is a spanning tree, there exist a path from $v_s$ to $v_t$ that lies in $T$, say $e_{j_1},\ldots,e_{j_{r}}$, $1\leq j_i\leq n_G$, for $i\leq r$. These edges together with the edge $e_j$ form a loop. By Proposition \ref{p6}, (\textbf{ii}), one now gets
\begin{equation}\label{c200}
\myphi^j_G=\sum_i\lambda_i \myphi^{j,j_i}_G \text{\;\;with\;} \lambda_i=\pm 1.
\end{equation} 
The Dodgson identity (\ref{c15}) for the symmetric matrix $P=P(\alpha)$ 
\begin{equation}\label{c201}
\det P^{i,i}\det P^{j,j}-\det P \det P^{ij,ij}=(\det P^{i,j})^2
\end{equation}
implies $\myphi^{j,i}_G\in \cI'$ for any $1\leq i\leq n_G$, $1\leq j\leq N_G$. By formula (\ref{c200}) above, one now gets $\myphi^j_G\in\cI'$ for $1\leq j\leq N_G$. 
\end{proof}

\begin{lemma}\label{lem15}
In terms of the matrix $P_G(\alpha)$, the singular locus $\Sing(Z_G)$ is given by 
\begin{equation}\label{c194}
\Sing(Z_G)=\left\{\alpha \in \AAA^{N_G}\;\big|\;  \rank P_G(\alpha)<n_G-1\right\}.
\end{equation}
\end{lemma}
\begin{proof}
Since the rank of a matrix is stable under the elementary row and column operations, Proposition \ref{lemma_diag} yields that it is enough to prove the statement for $P(\alpha):=\widetilde{P}(\alpha)$ with variables ordered in the way $T$ being a spanning tree formed by $e_1,\ldots,e_{n_G}$. Consider $t \in \Sing(Z_G)$. It follows that $\det P^{i,i}(t)=\partial_{\alpha_i} \myphi_G(t) = 0$ for $i=1,\ldots,n_G$ and $\myphi_G(t)=0$. The Dodgson identity (\ref{c201}) now implies  $\det P^{i,j}(t)=0$ for $i,j=1,\ldots,n_G$. Hence $\rank P(t)<n_G-1$.

For the opposite inclusion in (\ref{c194}), consider a point $t$ of the set on the right hand side. Since $\rank P(t)<n_G-1$, we get $\myphi_G(t)=\det P(t)=0$ and $ \myphi^i_G(t)=\det \widetilde P^{i,i}(t)=0$ for $i=1,\ldots,n_G$. Proposition \ref{p13} yields $t\in \Sing(Z_G)$.
\end{proof}

\section{$[Z_G]$ and $[Sing(Z_G)]$ in $K_0(Var_k)$}
The main theorems of this article concern the relations between the number of $\FF_q$-rational points of certain varieties. Nevertheless, the part of the computations are valid for $K_0(Var_k)$ that is more likely from the geometric point of view. 

For a fixed field $k$, the \emph{Grothendieck ring of varieties} $K_0(Var_k)$ is defined as a free $\ZZ$ module generated by the isomorphism classes [X] of separated schemes $X$ of finite type over $k$ modulo the following relation: $[X]=[Y]+[X\backslash Y]$ for closed subschemes $Y\subset X$. The ring structure is given by the product $[X]\cdot [X']=[(X\times Y)_{red}]$. The element 1 in this ring is $\1:=[\Spec k]$ and the Lefshetz element is defined by $\LL:=[\AAA^1_k]$. We will work with affine schemes and we usually write $[f_1,\ldots,f_r]$ (resp. $[\cI]$) for the class of $\cV(f_1,\ldots,f_r)\subset \AAA_k^N$ (resp. $\cV(\cI)\subset \AAA_k^{N}$) in $K_0(Var_k)$, where $f_1,\ldots,f_r$ is a collection of polynomials in $k[x_1,\ldots,x_N]$ (resp. $\cI\subset k[x_1,\ldots,x_N]$).

The graph polynomials $\Psi_G$ and $\myphi_G$ are linear with respect to each of the variables, as well as some of the Dodgson polynomials in certain situations. Recall the standard tool for computing the class in the Grothendieck ring using linearity (see \cite{BrSch}, Lemma 16):
\begin{lemma}\label{red_lemma}
Let $f^1,f_1,g^1,g_1\in k[\alpha_2,\ldots,\alpha_N]$. Then, for the varieties in the LHS in $\AAA^{N}$ and the ones of the RHS in $\AAA^{N-1}$, the following holds:
\begin{itemize}
\item[\textbf{i).}] $[f^1\alpha+f_1]=[f^1,f_1]\LL + \LL^{N-1}-[f^1]$. 
\item[\textbf{ii).}] $[f^1\alpha+f_1,g^1\alpha+g_1]=[f^1,f_1,g^1,g_1]\LL+[f^1g_1-g^1g_1]-[f^1,g^1]$.
\end{itemize} 
\end{lemma}

\begin{proposition}\label{Prop1} Let $G$ be a graph with $h_G\geq 2$. 
Then in $K_0(Var_k)$
\begin{equation}
[\myphi_G]\equiv 0 \mod \LL^2.
\end{equation} 
\end{proposition}
\begin{proof}
The proof is similar to the proof of Proposition 18 of \cite{BrSch}. 
By Euler's formula, the condition $h_G\geq 2$ is equivalent to $n_G+2\leq N_G$, and $n_G$ is the degree of $\myphi_G$. If $G$ is disconnected, then $\myphi_G=0$ and there is nothing to prove. Assume $G$ is a connected graph.
Using induction on $r$, we prove that for $f\in \ZZ[\alpha_1,\ldots,\alpha_r]$ of degree $\leq r$, and for any $G$ with at least 2 loops and any edge of $G$, say $e_1$, there exist elements $a(f),b(G,1),c(G)\in K_0(Var_k)$ such that
\begin{itemize}
\item[\textbf{1.}] $[f]=a(f)\LL \mod \LL^2$.
\item[\textbf{2.}] $[\myphi_{G,1},\myphi^1_G]=b(G,1)\LL\mod \LL^2$.
\item[\textbf{3.}] $[\myphi_G]=c(G)\LL^2 \mod \LL^3$.
\end{itemize}
1). For $r=1$ the statement is obvious. By Lemma \ref{red_lemma}, (\textbf{i}), for $f=f^1\alpha_1+f_1$, one computes $[f]=\LL^{r-1}-[f^1]+[f^1,f_1]\LL$.  Since the degree of $f^1$ is also less then the number of variables, we can construct $a(f)$ inductively: 
\begin{equation}
a(f):=[f^1,f_1]-a(f^1).
\end{equation} 
2). 
Fix any other edge $e_2$. By contraction-deletion formula (\ref{c10}) for the graphs $G\backslash 1$ and $G\q 1$,
$\myphi^1_G=\myphi^{12}_G\alpha_2+\myphi^1_{G,2}$ and $\myphi_{G,1}=\myphi^2_{G,1}\alpha_2+\myphi_{G,12}$. The Dodgson identity (\ref{c16}) reads $\myphi^1_{G,2}\myphi_{G,1}^2-\myphi^{12}_{G}\myphi_{G,12}= (\myphi^{1,2}_{G})^2$. Lemma \ref{red_lemma} implies
\begin{equation}
[\myphi^1_G,\myphi_{G,1}]=\LL[\myphi^1_{G,2},\myphi_{G,1}^2,\myphi^{12}_{G},\myphi_{G,12}]+[\myphi^{1,2}_{G}]-[\myphi^{12}_G,\myphi^2_{G,1}].
\end{equation}
Note that $\deg \myphi^{1,2}_{G}=n_G-1\leq N_G-3$, thus $\myphi^{1,2}_{G}$ satisfies the conditions in for part \textbf{1}. For positive $n_G$, we inductively define
\begin{equation}
b(G,1):=[\myphi^1_{G,2},\myphi_{G,1}^2,\myphi^{12}_{G},\myphi_{G,12}] + a(\myphi^{1,2}_{G}) - b(G\q 2,1),
\end{equation}
where the choice $e_2$ on each step is made in the way to avoid the contraction of a self-loop. The base of the induction is a graph with one vertex and $N_G-n_G\geq 2$ self-loops. Then $\myphi_{G\backslash 1}=0$ and $\myphi_{G\q 1}=1$. One gets $b(G,1)=\1$ for $N_G-n_G= 2$ and $b(G,1)=0$ for $N_G-n_G> 2$.\\ 
3). Since $\myphi_G$ is linear in $\alpha_1$,
 Lemma \ref{red_lemma}, \textbf{(i)} implies
\begin{equation}
[\myphi_G]=[\myphi^1_G,\myphi_{G,1}]\LL + \LL^{N_G-1}-[\myphi^1_G].
\end{equation} 
If $n_G\geq 2$, define $c(G)$ inductively by
\begin{equation}
c(G):=b(G,1)-c(G\q 1). 
\end{equation}
If $G$ less than 2 vertices, then $G$ should be formed by 1 vertex and $\ell\geq 2$ self-loops. One again computes $c(G)=1$ for $\ell=2$, and $c(G)=0$ otherwise.
\end{proof}
In $K_0(Var_k)$, there are not only the zero-divisors, but also elements $z$ such that $z\LL=0$. That is why the element $c(G)$ above is defined only modulo the ideal $Ann_{K_0}(\LL)$ generated by such elements $z$.  
\begin{definition}\label{def3}
Define by $\widetilde{\LL}=\langle\LL\rangle+Ann_{K_0}(\LL)\subset K_0(Var_k)$ to be the ideal generated by $\LL$ and the elements of $Ann_{K_0}(\LL)$.
For a graph $G$ define the invariant $c_2^{dual}(G)$ to be the element $c(G)$ from the proof above. In other words,
\begin{equation}
c_2^{dual}(G):=[\myphi_G]/\LL^2 \mod \widetilde{\LL}.
\end{equation}
\end{definition}

If one of the loops of $G$ is of length 2,
using (\ref{e101}), one can easily prove that $c_2^{dual}(G)\equiv 0 \mod \widetilde{\LL}$  since we can get rid of one of the variables and get a fibration with each fibre isomorphic to $\AAA^1$.

In the case $G$ has a loop of length 3, we are able to give a concrete description of the $c_2^{dual}(G)$ invariant.
\begin{proposition}\label{p4}
Let $G$ be a graph with 3 edges (say $e_1$,$e_2$,$e_3$) forming a triangle and with $h_G\geq 3$. Then 
\begin{equation}
c^{dual}_2(G) \equiv [\myphi^{1,2}_{G,3},\myphi^{13,23}_G] \mod \widetilde{\LL}. 
\end{equation}  
\end{proposition}
\begin{proof}
Recall that the proof of the corresponding statement for the graph polynomial uses the special structure of $\Psi_G$ in the case of the existence of a 3-valent vertex, see Lemma 24 in \cite{BrSch}. There is also a formula for $\Psi_G$ in the case of the existence of a triangle in $G$, it can be found in Example 33, \cite{Br}: 
\begin{multline}
\Psi_G= f^{123}\alpha_1\alpha_2\alpha_3+(f^1+f^2)\alpha_1\alpha_2+(f^1+f^3)\alpha_1\alpha_3+ (f^2+f^3)\alpha_2\alpha_3\\
+f^0(\alpha_1+\alpha_2+\alpha_3),
\end{multline}
together with $f^0f^{123}=f^1f^2+f^2f^3+f^1f^3$, where $f^{123}=\Psi^{123}$, $f^0=\Psi^i_{jk}$, $f^i=\Psi^{ij,ik}$ for $\{i,j,k\}=\{1,2,3\}$.
We dualize this using (\ref{c16}) to get a convenient formula for $\myphi_G$:
\begin{multline}\label{c90}
\myphi_G=g_0(\alpha_1\alpha_2+\alpha_2\alpha_3+
\alpha_1\alpha_3)+(g_1+g_2)\alpha_3+(g_1+g_3)\alpha_2\\
\quad+ (g_2+g_3)\alpha_1+g_{123},\quad
\end{multline}
with the only identity 
\begin{equation}\label{c91}
g_0g_{123}=g_1g_2+g_2g_3+g_1g_3.
\end{equation}
Here $g_{123}=\myphi_{123}$, $g_0=\myphi^{ij}_k$, $g_i=\myphi^{j,k}_i$, $g_i+g_j=\myphi^k_{ij}$, $\{i,j,k\}=\{1,2,3\}$. The formula looks identical to that for $\Psi_G$ in the case $G$ has a 3-valent vertex (see \cite{Br}, Example 32), so one can use the same strategy as in the proof of Proposition 23 in \cite{BrSch} to derive
\begin{equation}\label{c93}
[\myphi_G]=\LL^{N-1}+\LL^3[g_0,g_1,g_2,g_3,g_{123}]- \LL^2[g_0,g_1,g_2,g_3].
\end{equation}
Thus $c_2^{dual}(G)\equiv [g_0,g_1,g_2,g_3]\mod \widetilde{\LL}$.
The next part of the proof goes similar as the proof of Lemma 24 in \cite{BrSch}. By (\ref{c91}), the inclusion-exclution formula yields 
\begin{equation}\label{c94}
[g_0,g_3]=[g_0,g_1g_2,g_3]=[g_0,g_1,g_3]+[g_0,g_2,g_3]
-[g_0,g_1,g_2,g_3],
\end{equation}
and $[g_0,g_1+g_3]=[g_0,g_1+g_3,g_1g_3]=[g_0,g_1,g_3]$. By contraction-deletion (\ref{c10}), $[g_0,g_1+g_3]=[\myphi^{12}_3,\myphi^2_{13}]=[\myphi^1_{G'}\myphi_{G',1}]$ for $G'=G\backslash 3\q 2$. Since $G$ has at least 3 loops, the graph $G'$ has $h_{G'}\geq 2$. We use Proposition \ref{Prop1}, \textbf{2} and get $\LL\big| [\myphi^1_{G'}\myphi_{G',1}]$. By symmetry, we can also get the divisibility $\LL|[g_0,g_2,g_3]$. Now (\ref{c93}) and (\ref{c94}) imply
\begin{equation}
[\myphi_G]\equiv \LL^2[g_0,g_1,g_2,g_3]\equiv\LL^2[g_0,g_3]\equiv \LL^2[\myphi^{1,2}_3,\myphi^{13,23}] \mod \widetilde{\LL}^3. 
\end{equation}
The statement follows from the definition of $c^{dual}_2(G)$.
\end{proof}

We are going to use Proposition 29 from \cite{BSY}. This is the simultaneous elimination of one variable from an ideal in the Grothendieck ring whose generators are all linear in that variable. 
\begin{proposition}\label{p9}
Let $f_1,\ldots,f_n$ are linear in $\alpha$, say $f_i=f_i^\alpha \alpha+f_{i,\alpha}$, $1\leq i\leq n$. Then
\begin{multline}\label{e10}
[f_1,\ldots,f_n]=[f^\alpha_1,f_{1,\alpha},\ldots,f^\alpha_n,f_{n,\alpha}]\LL +\\ [[f_1,f_2]_\alpha,\ldots,[f_1,f_n]_\alpha] - [f^\alpha_1,\ldots,f^\alpha_n]\\
\sum^{n-2}_{k=1}([f^\alpha_1,f_{1,\alpha}\ldots,f^\alpha_k,f_{k,\alpha},[f_{k+1},f_{k+2}]_{\alpha},\ldots,[f_{k+1},f_n]_{\alpha}]\\
-[f^\alpha_1,f_{1,\alpha}\ldots,f^\alpha_k,f_{k,\alpha}]).
\end{multline}
\end{proposition}

Now we return to the singular locus of the dual graph hypersurface $\Sing(Z_G)$ appeared in Definition \ref{def7}. In the Grothendieck ring one immediately gets
\begin{equation}
[\Sing(Z_G)]=[\myphi_G,\myphi_G^1,\ldots,\myphi_G^{N_G}]\in K_0(Var_k).
\end{equation}
\begin{proposition}
Let $G$ be a connected graph with $N=N_G$ edges and with $h_G\geq 2$ loops. Then in $K_0(Var_k)$ one has
\begin{equation}\label{e11}
[Sing(Z_G)]+[Sing(Z_{G\q 1})]=\LL[\myphi^1,\myphi_1,\{\myphi^{1t},\myphi^t_1\}_{t=2,\ldots,N}]+[\myphi^1,\myphi_1]
\end{equation}
for some edge $e_1$.
\end{proposition}
\begin{proof}
The proof is very similar to the proof of Lemma 30 in \cite{BSY}. The edge $e_1$ is chosen to be an edge which deletion does not disconnect $G$. We write $[Sing(Z_G)]=[\myphi,\myphi^1,\ldots, \myphi^N]$ and  apply Proposition (\ref{p9}) to the set of polynomials $\myphi,\myphi^1,\ldots, \myphi^N$ linear in the variable $\alpha=\alpha_1$. Each summand of the big sum on the right hand side in (\ref{e10}) is of the form
\begin{equation}
\big[\myphi^1\!,\myphi_1,\ldots,\myphi^{1t}\!,\myphi^t_1, [\myphi^{t+1}\!,\myphi^{t+2}]_1,\ldots,[\myphi^{t+1}\!,\myphi^{N}]_1 \big]-[\myphi^1\!,\myphi_1,\ldots,\myphi^{1t}\!,\myphi^t_1].
\end{equation}
By Corollary \ref{cor7}, for any $a\neq b\in I_N\backslash 1$, the resultant $[\myphi^a,\myphi^b]_1$ is contained in the radical of the ideal spanned by $\myphi^1,\myphi_1$. In the Grothendieck ring we see only the reduced scheme structure (an ideal is undistinguishable from its radical). It follows that the two classes above sum to 0 for every $t$. Hence (\ref{e10}) reduces to 
\begin{equation}\label{e15}
[\Sing(Z_G)]=\LL[\myphi^1,\myphi_1,\ldots,\{\myphi^{1t},\myphi^t_1\}_t]-[\myphi^1,\{[\myphi,\myphi^t]_1\}_t]-[\myphi^1,\{\myphi^{1t}\}_t],
\end{equation} 
where $t$ ranges from 2 to $N$ in each of the three expressions on the right hand side. Since $[\myphi,\myphi^t]_1=\myphi^1\myphi^t_1-\myphi_1\myphi^{1t}$, the middle summand on the right hand side simplifies as $[\myphi^1,\{[\myphi,\myphi^t]\}_t]=[\myphi^1,\{\myphi_1\myphi^{1t}\}_t]$. Considering the cases $\myphi_1=0$ and $\myphi_1\neq 0$ separately, one computes
\begin{multline}\label{e16}
[\myphi^1,\{\myphi_1\myphi^{1t}\}_t]=[\cV(\myphi^1,\{\myphi_1\myphi^{1t}\}_t)\backslash\cV(\myphi_1,\myphi^1,\{\myphi_1\myphi^{1t}\}_t)]+\\
[\myphi_1,\myphi^1,\{\myphi_1\myphi^{1t}\}_t]=[\cV(\myphi^1,\{\myphi^{1t}\}_t)\backslash\cV(\myphi_1,\myphi^1,\{\myphi^{1t}\}_t)]+[\myphi_1,\myphi^1]\\ =[\myphi_1,\myphi^1]+[\myphi^1,\{\myphi^{1t}\}_t]-[\myphi^1,\myphi_1,\{\myphi^{1t}\}_t].
\end{multline}
Now we can consider a corolla in $G$ which contains the edge $e_1$ and we apply Corollary \ref{cor9}. It follows that $\myphi_1\in Rad(\cI)$ for the ideal $\cI\subset\ZZ[\{\alpha_i\}_{I\backslash 1}]$ generated by $\myphi^1$, and $\{\myphi^{1i}\}_{i\in I\backslash 1}$ for some $I\subset I_N$. Thus the second and the third summand on the last expression in (\ref{e16}) sum up to zero. The last term on the right in (\ref{e15}) defines the singular locus of the dual graph hypersurface for the graph $G\q 1$. 
\end{proof}

\begin{theorem}\label{thm20}
Let $G$ be a graph with at least 2 loops. Then for the singular locus of the dual graph hypersurface of $G$, the following congruence holds:
\begin{equation}\label{e19}
[\Sing(Z_G)]\equiv 0 \mod \LL.
\end{equation}
\end{theorem}
\begin{proof}
If $G$ is disconnected then $\myphi_G=0$ and there is nothing to proof.

 If $G$ has a self-loop, say formed by an edge $e_1$, then by (\ref{e100}) all the $\myphi^I_J$ for $G$ are independent of $\alpha_1$. It follows that we can project down to the situation for $G\backslash 1$ with fibres $\AAA^1$, the statement follows. 

If $G$ has a loop of length 2, then by (\ref{e101}), one can write $\myphi_G=\myphi_{G\backslash 1\q 2}(\alpha_1+\alpha_2)+\myphi_{G\backslash 12}$. 
After the changing of the variables $\alpha_2:=\alpha_1+\alpha_2$, we can again project to the situation for $G\backslash 1$ with fibres $\AAA^1$ and (\ref{e19}) holds. 

So we can assume that the graph $G$ is connected with no self-loops or double edges. The proof goes by  the induction on the number of edges $N_G$. The assumptions on $G$ imply $N_G\geq 5$.  Since $h_G\geq 2$ is equivalent to $n_G+2\leq N_G$ by Euler's formula, we are able to use Proposition \ref{Prop1}, \textbf{2}  and we get 
$[\myphi^1_G,\myphi_{G,1}]\equiv 0 \mod \LL$. Hence, (\ref{e11}) implies
\begin{equation}
[Sing(Z_G)]\equiv -[Sing(Z_{G\q 1})] \mod \LL.
\end{equation}
If the graph $G\q 1$ still has a double edge then the divisibility $\LL|[Sing(Z_{G\q 1})]$ is clear. Otherwise we proceed by induction.
\end{proof}

\section { The $c_2$ invariant in position space} 
Fix a field $k$ ($k$ can be $\FF_q$, $\CC$ or (the usual for physicists) $\RR$). For the convenience of the computation, we work not with Euclidian metric, but with the metric defined by
\begin{equation}
|x|^2=x^1x^2+x^3x^4,\quad \text{for}\;\; x=(x^1,x^2,x^3,x^4)\in k^4.
\end{equation}
Consider a log-divergent graph $G$ with $N_G$ edges $\{e_i\}_{i\in I_N}$ and $n_G+1$ vertices. To each vertex we associate a variable $x_p$, $p=1,\ldots,n+1$, with $n:=n_G$. The propagator attached to an edge $e_i$ with endpoints with variables $x_s$ and $x_t$ is of the form
\begin{equation}\label{d13}
\frac{1}{q_i(x)}=\frac{1}{|x_s-x_t|^2}\in Frac(\ZZ[\{x^j_p\}_{p,j}]),
\end{equation} 
with $1\leq i\leq N_G$, $1\leq j\leq 4$, $1\leq p\leq n+1$ and with one exception: $x_{n+1}$ is set to be zero in any expression above where it appears, i.e. in the case when $e_i$ is incident to $(n+1)$-th vertex. We need this restriction to define the period.

For a primitive log-divergent graph $G$, $N_G=2n$, the \emph{Feynman period} representation in the position space is defined to be the value
\begin{equation}
I_G^{pos}:=\int_{\PP\RR^{4n-1}} \frac{\Omega(x)}{q_1\ldots q_{N_G}},
\end{equation}
where $\Omega(x)$ is the standard differential form in projective space with coordinates all of the $x^j_p$, $1\leq j\leq 4$, $1\leq p\leq n$. We will be interested in the configuration of the quadrics $q_i$ in $\AAA^{4n}_k$. One can easily translate the results from projective space to affine one and vice versa; for counting of $\FF_q$-rational points we prefer the affine setting. 

Consider the \emph{universal quadric} 
\begin{equation}\label{d15}
\cQ (\alpha,x)=\sum_{i=1}^{N_G} \alpha_iq_i(x)\in \ZZ\big[\{\alpha_i\}_{i\in I_N},\{x_p\}_{p=1,\ldots,n}\big]
\end{equation}
depending on the edge (Schwinger) variables $\alpha_i$ and the vertex variables (4-vertors) $x_p$. This is the key tool of the Schwinger trick, see Figure 1. 

We return to (\ref{d13}) and consider two adjacent vertices with associated variables $a$ and $b$. The denominator of the propogator can be written as
\begin{equation}\label{d19}
|a-b|^2=\left( a^2 a^4\, b^2\, b^4 \right) 
\left(
\begin{array}{cccc}
1 &0 &-1 &0 \\
0 & 1 & 0 & -1\\
-1 &0 &1 &0 \\
0 & -1 & 0 & 1\\
\end{array}
\right)
\left(
\begin{array}{c}
a^1 \\
a^3 \\
b^1 \\
b^3 
\end{array}
\right).
\end{equation}
It follows that the universal quadric (\ref{d15}) can be written as coming from a matrix consisting of blocks of the shape (\ref{d19}) multiplied by $\alpha_i$s. After a suitable permutation  of rows and columns, one gets
\begin{equation}\label{d20}
\cQ(\alpha,x)=
\left(
\begin{array}{c}
x^2 \\
x^4
\end{array}
\right)^{\!t}
\left(
\begin{array}{cc}
P_G(\alpha) &0\\
0 & P_G(\alpha)
\end{array}
\right)
\left(
\begin{array}{c}
x^1 \\
x^3
\end{array}
\right),
\end{equation}
where $x^j$ is a vector build up of consecutive coordinates $x_1^j,\ldots,x_n^j$, $1\leq j\leq 4$, and $P_G(\alpha)\in \Mat_{n, n}(\ZZ[\{\alpha_i\}_{i\in I_N}])$ is the matrix from (\ref{d21}).

Recall that in Proposition \ref{lemma_diag} we have constructed the matrix $\widetilde{P}(\alpha)$ out of $P_G(\alpha)$ by the diagonalization with respect to the edges of a given fixed spanning tree $T$ of $G$. We need two more propositions.

\begin{proposition}\label{p14}
For a graph $G$ with $N_G$ edges and $n+1$ vertices, and for a subset of edges $I\subset I_N$, define by $P_{\bar{I}}$ the matrix $\widetilde{P}(\alpha)|_{\alpha_i=0,i\not\in I}$ that is obtained from $\widetilde{P}(\alpha)$ by setting to zero all the variables with indexes in $I_N\backslash I$. Then 
\begin{equation}
(\LL-1)\big(\LL^{|I|-1}\big[\{q_i(x)\}_{i\in I}\big]-\LL^{2n-1}\big[P_{\bar{I}}\cdot x^2,P_{\bar{I}}\cdot x^4\big]\big)=0 
\end{equation}
where $[\{q_{i}(x)\}_{i\in I}]$ denotes the class of the vanishing of all the $q_i$s (with $i$ from the given set)
in $K_0(Var_k)$.
\end{proposition}
\begin{proof}
We compute the number of points on the quadric $\cQ_I(\alpha,x)=\sum_{i\in I}\alpha_i q_i(x)$ in two ways projecting to the space of the edge variables $\alpha$ or of the vertex variables $x$. Firstly, consider the projection of $\cQ_I$ to $\AAA^{|I|}(\{x^j_p\})$, $1\leq j\leq 4$, $1\leq p\leq n$. Since $\cQ_I$ is linear in each $\alpha_i$, the general fibre is isomorphic to $\AAA^{|I|-1}$. In the case of the intersection of all the quadrics $q_i$ (writing $[\{q_i\}_{i\in I}]$ for the class in the Grothendieck ring in this situation), the fibre is isomorphic to $\AAA^{|I|}$. We get
\begin{equation}\label{c209}
[\cQ_I]=\LL^{|I|-1}\left(\LL^{4n}-[\{q_i(x)\}_{i\in I}]\right) + \LL^{|I|}[\{q_i(x)\}_{i\in I}].
\end{equation}
\\
On the other hand, comparing to (\ref{d20}), $\cQ_I(\alpha,x)$ can be rewritten in the form 
\begin{equation}\label{c210}
\cQ_I(\alpha,x)=
\left(
\begin{array}{c}
x^2 \\
x^4
\end{array}
\right)^{\!t}
\left(
\begin{array}{cc}
P_{\bar{I}}(\alpha) &0\\
0 & P_{\bar{I}}(\alpha)
\end{array}
\right)
\left(
\begin{array}{c}
x^1 \\
x^3
\end{array}
\right). 
\end{equation}
and thus defines a fibration over $\AAA^{|I|+2n}(\alpha_i,x^1_p,x^3_p)$, $i\in I$, $1\leq p\leq n$ with fibres linear subspaces in the variables $x^2_p$ and $x^4_p$. One computes 
\begin{equation}
[\cQ_I(\alpha,x)]=\LL^{2n-1}(\LL^{|I|+2n}-[P_{\bar{I}}\cdot x^2,P_{\bar{I}}\cdot x^4])]) + \LL^{2n}[P_{\bar{I}}\cdot x^2,P_{\bar{I}}\cdot x^4].
\end{equation}
Together with (\ref{c209}) this yields the statement.
\end{proof}
\begin{proposition}\label{p15}
Define $\myphi_{G,\bar{I}}:=\det P_{\bar{I}}(\alpha)=\myphi_G|_{\alpha_i=0,i\not\in I}$ for $I\subset I_N$. Then 
\begin{equation}\label{c213}
[P_{\bar{I}}\cdot x^2,P_{\bar{I}}\cdot x^4]\equiv \LL^{|I|} + (\LL^2-1)[\myphi_{G,\bar{I}}]-\LL^2[\rank P_{\bar{I}}<n_G-1] \mod \LL^4.
\end{equation}
\end{proposition}
\begin{proof}
The equation $P_{\bar{I}}\cdot x^2=0$ is a system of $n$ linear equations in the variables $x^2$, thus the vanishing locus of this system is isomorphic to $\AAA^r$ for $r=\corank P_{\bar{I}}$. The equation $P_{\bar{I}}\cdot x^4=0$ gives the same system but in the variables $x^4$. It follows that 
\begin{multline}
[P_{\bar{I}}\cdot x^2,P_{\bar{I}}\cdot x^4]\equiv [\corank P_{\bar{I}}=0]+\LL^2[\corank P_{\bar{I}}=1] \mod \LL^4\equiv \\  \LL^{|I|}-[\corank P_{\bar{I}}>0]+\LL^2\big([\corank P_{\bar{I}}>0]-[\corank P_{\bar{I}}>1]\big) \mod \LL^4.
\end{multline}
Since $(\myphi_G=0)\Leftrightarrow (\corank P_{\bar{I}}>0)$, the congruence (\ref{c213}) follows.
\end{proof}

From now on we need to reduce to the computation of the number of rational points over finite fields. 

Consider $f_1,\ldots,f_r\subset \ZZ[a_1,\ldots,a_N]$ and fix $q=p^s$ a prime power. Denote by $\bar{f}_i$ the reduction of $f_i$ modulo $q$. Define $[f_1,\ldots,f_r]_q\in \NN_{0}$ to be the number of $\FF_q$-rational points of the variety $\cV(\bar{f_1},\ldots,\bar{f_r})\subset \AAA^N_{\FF_q}$. 

Similarly to what happens in momentum space, our object of interest is the point counting function of the union $\cV(q_1\ldots q_N)$ of quadrics that is the denominator of the differential form in the representation of a period in position space.
We are going to use Chevalley-Warning theorem. The possible analogue of this result in the Grothendieck ring of varieties is called the geometric Chevalley-Warning question and was recently proved to be false (see \cite{Hu}). This means that the results for the counting points functions over $\FF_q$ below cannot be easily lifted to the Grothendieck ring. 

The counting points functor factors through the Grothendieck ring of varieties mapping $\1$ to 1 and $\LL$ to $q$, so the results of the previous two propositions and the results of Section 2 imply the corresponding congruences for the number of rational points. For instance, the following definition corresponds to Definition \ref{def3} and will be used later in the section.
\begin{definition}\label{def23}
For a graph $G$ with $h_G\geq 2$ and a prime power $q$, define the invariant $c_2^{dual}(G)_q$ by 
\begin{equation}
c_2^{dual}(G)_q:=[\myphi_G]_q/q^2 \mod q.
\end{equation}
\end{definition}

\begin{theorem} (Chevalley-Warning) \label{T18}
Let $f_1,\ldots,f_r\in \ZZ[a_1,\ldots,a_N]$ be polynomials with $\sum_i \deg f_i< N$. Then for any prime power $q$,
\begin{equation}\label{cw-thm}
[f_1,\ldots,f_r]_q\equiv 0 \mod q.
\end{equation}
\end{theorem}

\begin{proposition}\label{p19} For any graph $G$ with $N_G\leq 2n_G$, one has 
\begin{multline}\label{c216}
[q_1\ldots q_{N_G}]_q \equiv (-q)^{2n-N_G}\big([\myphi_G]_q+q^2[\Sing(Z_G)]_q\\ - q \sum_{i\in I_N}[\myphi_{G\backslash i}]_q + q^2\sum_{i,j\in I_N}[\myphi_{G\backslash i,j}]_q\big)\mod q^3.
\end{multline}
\end{proposition}
\begin{proof}
First we apply the inclusion-exclusion formula
\begin{equation}\label{c218}
[q_1\ldots q_{N_G}]_q=\sum_{I\subset I_N} (-1)^{|I|+1}[\{q_i\}_{i\in I}]_q.
\end{equation}
Proposition \ref{p14} implies $[\{q_i\}_{i\in I}]_q= q^{2n-|I|}[P_{\bar{I}}\cdot x^2,P_{\bar{I}}\cdot x^4]_q$. We immediately get $q^3|[\{q_i\}_{i\in I}]_q$ for $|I|\leq N-3$.   For each $I$ in the case $N-2\leq |I|\leq N$, we are going to use Proposition \ref{p15}. It follows that $q^{2n-|I|}q^2[\rank P_{\bar{I}}<n_G-1]_q\equiv 0 \mod q^3$ for $|I|<N$, and $[\myphi_{G,\bar{I}}]_q=[\myphi_{G\backslash (I_N\backslash I)}]_q$. In the case $|I|=N$ one obtains $I=I_N$, $[\myphi_{G,\bar{I}_N}]_q=[\myphi_G]_q$, and $[\rank P_{\bar{I}}<n-1]_q=[Sing(Z_G)]_q$, which follows from Lemma \ref{lem15}. Thus  
\begin{equation}
\big[\{q_i\}_{i\in I}\big]_q \equiv \left\{
\begin{aligned}
&-q^{2n-N}([\myphi_G]_q+q^2[\Sing(Z_G)])\mod q^3,\;\; I=I_N,\\
&-q^{2n-|I|}([\myphi_{G,\bar{I}}]_q)\mod q^3,\;\; |I|=N_G-1, \,N_G-2,\\
& \;0 \mod q^3, \;\;  |I|\leq N_G-3.
\end{aligned}\right.
\end{equation}
Summing everything together using (\ref{c218}), one gets (\ref{c216}).
\end{proof}
\begin{corollary}
For $G$ a graph with $N_G\leq 2n_G$, $n_G\geq 2$ one has
\begin{equation}\label{c220}
[q_1\ldots q_{N_G}]_q\equiv 0 \mod q^2.
\end{equation}
\begin{proof}
Proposition \ref{p19} trivially implies the statement for $2n_G>N_G+1$, so we need to take care of the cases $2n_G=N_G+1$ and $2n_G=N_G$, $n_G\geq2$. By Proposition \ref{Prop1}, $q^2\big| [\myphi_G]_q$ for $2n_G=N_G$ and $n_G\geq 2$, and $q\big| [\myphi_G]_q$ for $2n_G=N_G+1$ and $n_G\geq 2$. For the third summand of the right hand side of (\ref{c216}) in the case $2n_G=N_G$, we have $q|[\myphi_{G'}]_q$ for any $G'=G\backslash e$ with $e\in E(G)$. Now (\ref{c220}) follows.
\end{proof}
\end{corollary}
\noindent Using this corollary, we can give the following definition.
\begin{definition}
Let $G$ be a graph with $N_G\leq 2n_G$ and $n_G\geq 2$. We define the $c_2$ invariant of $G$ in position space as follows:
\begin{equation}
c_2^{pos}(G)_q:=[q_1\ldots q_N]_q/q^2 \mod q^3.
\end{equation}
\end{definition}
Now we are able to prove the coincidence of $c_2$ invariants in the dual parametric space (Definition \ref{def23}) and in position space.
\begin{theorem}
Let $G$ be a graph with $n_G \geq 3$. Then the following holds.
\begin{enumerate}
\item If $N_G<2n_G$, then $c_2^{pos}(G)_q=0$.
\item If $N_G=2n_G$ (i.e. $G$ is log-divergent), then 
\begin{equation}
c_2^{dual}(G)_q=c_2^{pos}(G)_q.
\end{equation}
\end{enumerate}
\end{theorem}
\begin{proof}
Part 1). We are going to use Formula (\ref{c216}). In the case $2n_G>N_G+2$ the statement holds for trivial reasons. \\
\noindent If $2n_G=N_G+2$, then $q|[\myphi_G]_q$ by Proposition \ref{Prop1} for $N_G\geq n_G+2$  and by direct computation for $n_G=3$ and $N_G=4$.\\
\noindent If $2n_G=N_G+1$, then $N_G\geq n_G+2$, thus again $q^2|[\myphi_G]_q$. For any edge $e_1$, we also have $G\backslash 1$ disconnected or $N_{G\backslash 1	}\geq n_{G\backslash 1} +1$, hence $q| [\myphi_{G\backslash 1}]_q$. The statement follows.\medskip

\noindent Part 2). We have $N_G=2n_G$, so either $G\q e$ is disconnected or $N_{G\q e}\geq n_{G\q e}+1$, hence $q^2| [\myphi_{G\q e}]_q$. Similarly, either $G\q e_1 e_2$ is disconnected or $N_{G\q e_1e_2}\geq n_{G\q e_1e_2}+1$, hence $q| [\myphi_{G\q e_1e_2}]_q$. Thus, Formula (\ref{c216}) reduces to
\begin{equation}
[q_1\ldots q_N]_q \equiv \big([\myphi_G]_q+q^2[\Sing(Z_G)]_q\\  \big)\mod q^3 .
\end{equation}
The statement follows from Theorem \ref{thm20} and the definitions of $c_2^{pos}(G)_q$ and $c_2^{dual}(G)_q$.
\end{proof}

\section{The $c_2$ invariant respects dualization}

In this section we prove the coincidence of $c_2(G)_q$ and $c_2^{dual}(G)_q$ for a subset of log-divergent graphs $G$ which we call duality admissible. 

We cannot use the proof of the statements from the end of the previous section for the corresponding statements for $c_2$ in the Grothindieck ring $K_0(Var_k)$ since we intensively apply Chevalley-Warning vanishing. We do not use $K_0(Var_k)$ in this section at all, but we again intensively use the notation $[Y]$ here meaning the point-counting function. More precisely, starting from now, we omit the index $q$ and write $[Y]$ for the number of $\FF_q$--rational points of an affine scheme $Y$ (or its reduction) over $\FF_q$ for a fixed prime power $q$. This will make the formulas more readable. 
We also define $[Y]'$ to be $[Y\cap (\GG_m)^N]$ for a fixed embedding of $(\GG_m)^N\hookrightarrow\AAA^N_k$ , where $Y\subset\AAA^N_k$ is an affine scheme. For instance, the function $f\mapsto[f]'$ counts the number of solutions of $f=0$ with non-zero coordinates. 

For example, since $\myphi^I_J:=\iota(\Psi^J_I)$ for any graph $G$ and any edges indexed by $I,J\subset I_N$, one has a bijection between non-zero solutions of $\Psi^I_J=0$ and non-zero solutions of $\myphi^J_I=0$, thus
\begin{equation}\label{c80}
[\myphi^I_J]'=[\Psi^J_I]'.
\end{equation}
Assume for a moment that $\Psi\in\ZZ[\alpha_1\ldots,\alpha_N]$ is any polynomial of degree $n$ linear with respect to each of the variable (not necessarily a graph polynomial).
Grouping the summands by the number of the variables $\alpha_i$ which are zero, we get
\begin{equation}\label{sum1}
[\Psi]=[\Psi]'+\sum_i[\Psi_i]'+\sum_{i,j}[\Psi_{i,j}]'+\sum_{i,j,k}[\Psi_{ijk}]'+\ldots=[\Psi]'+\sum_{t=1}^{N} \sum_{|I|=t}[\Psi_I]'.
\end{equation}
On the other hand, computing affinely, in the solutions for a summand $[\Psi_I]$ the variables $\alpha_j$, $j\notin I$ are allowed to vanish. By inclusion-exclusion, one obtains 
\begin{equation}\label{sum2}
[\Psi]=[\Psi]'+\sum_i[\Psi_i]-\sum_{i,j}[\Psi_{i,j}]+\sum_{i,j,k}[\Psi_{ijk}]-\ldots=[\Psi]'+\sum_{t=1}^N(-1)^{t+1}\sum_{|I|=t} [\Psi_{I}].
\end{equation}
We should restrict our attention to the following type of graphs. 

\begin{definition}\label{def26}
A log-divergent graph $G$ with $h_G,n_G\geq 3$ and $N=N_G$ edges is called $\underline{duality\; admissible}$ if 
\begin{equation}
[\myphi^J_{I}]\equiv 0 \mod q^3
\end{equation}
for any $I,J\subset I_N$ with $|J|>|I|\geq 0$, $|I|\leq n_G-3$.
\end{definition}

The motivation of this definition is the observation that the similar conditions for the graph polynomial itself are satisfied, and both congruences will be used in the proof of the main theorem.
\begin{proposition} \label{lemma8}
Let $G$ be a log-divergent graph with $N=N_G$ edges. Then 
\begin{equation}
[\Psi^I_{J}]\equiv 0 \mod q^3
\end{equation}
for any $I,J\subset I_N$ with $|I|>|J|\geq 0$, $|J|\leq n_G-3$.
\end{proposition}
\begin{proof}
1). We can assume $G$ is connected, otherwise the divisibility is clear. Since $G$ in log-divergent, $G$ has $(N_G,h_G,n_G)=(2n,n,n)$. We know $\Psi^I_{G,J}=\Psi_{G'}$ for the graph $G':=G\backslash I\q J$. Again, assume $G'$ is connected. Each deletion of an edge of $G$ decreases $h_G$, and each contraction of an edge decreases $n_G$. Thus $G'$ has $(N_{G'},h_{G'},n_{G'})= (2n-|I|-|J|,n-|I|,n-|J|)$. If $G'$ has a vertex of degree 1 with an incident edge $e_1$, then $\Psi_{G'}$ is independent of $\alpha_1$ and one computes $[\Psi_{G'}]=q[\Psi_{G''}]$ for $G'':=G'\q 1$. The divisibility $q^2|[\Psi_{G'\q 1}]$ is standard, follows from the analogue of Proposition \ref{Prop1}, see \cite{BrSch}, Lemma 16. Now one gets $q^3|[\Psi_{G'}]$. If $G'$ has a 2-valent vertex with incident edges $e_1$ and $e_2$, then, after the change of the variables $\alpha_2:=\alpha_1+\alpha_2$ one gets rid of $\alpha_1$ and obtains $[\Psi_{G'}]=q[\Psi_{G''}]$ for $G'':=G'\q 1$ (see \cite{BrSch}, Lemma 17, (1)). Thus $q^3|[\Psi_{G'}]$ in this case.

Consider now the case when all the vertices of $G'$ are of degrees $\geq 3$. Since $G$ is log-divergent, there should exist a vertex of $G'$ of degree 3. Indeed, $N_G=2n$ and $|I|>|J|$ imply $N_{G'}<2n_{G'}$. But on the other hand, each vertex is incident to $\geq 4$ edges and each edge is counted twice, so $2(n_{G'}+1)\leq N_{G'}$, a contradiction. 

If now $G'$ has a $J\leq n-3$, then $n_{G'}\geq 3$ and Lemma 24 in \cite{BrSch} gives us $[\Psi_{G'}]\equiv q^2[\Psi^{1,2}_{G',3},\Psi^{13,23}_{G'}]\mod q^3$. Since $2h_{G'}<N_{G'}$, we apply Chevalley-Warning (Theorem \ref{T18}) to the polynomials in the last square brackets and get $[\Psi_{G'}]\equiv 0 \mod q^3$.
\end{proof}

\begin{proposition}\label{prop33}
Let $G$ be a graph with $h_G,n_G\geq 3$. Assume that for any subsets of edges of $G$ indexed by $I,J\subset I_N$, $|I|<|J|$, for the subquotient graph $G\backslash I\q J$ the following holds: $G\backslash I\q J$ is disconnected, or is planar, or has a loop of length at most 3. Then $G$ is duality admissible. 
\end{proposition}
\begin{proof}
2) Let $G':=G\backslash I\q J$ again in the way that $\myphi^J_{G,I}=\myphi_{G'}$. Instead of the vertices of small degree, we look at loops of small length. Similarly to the prove above, we consider the cases of the existence of a self-loop or a double edge (2-loop) and use (\ref{e100}), (\ref{e101}), and Proposition \ref{Prop1} and easily get $q^3|[\myphi_{G'}]$.

Now consider the case when all the loops of a $G'$ are of length at least 3. Assume $G'$ is planar. There is a notion of the planar dual graph $\gamma^{dual}$ of a planar graph $\gamma$, 
(see, for example, (2.2) in \cite{Sch}). Its vertices (resp. cycles) correspond to cycles (resp. vertices) of the original graph, $h_{\gamma^{dual}}=n_{\gamma}$ and $n_{\gamma^{dual}}=h_{\gamma}$.  The important identity is $\myphi_{\gamma}=\Psi_{\gamma^{dual}}$. Thus, one can use the statement of Proposition \ref{lemma8} and derive $[\myphi_{G'}]\equiv 0\mod q^3$.

The last case to consider is $G'$ has no self-loops or 2-loops and is not planar. By the assumption, $G$ is duality admissible, so $G'$ should have a loop of length 3 (say, formed by edges $e_1$, $e_2$ and $e_3$). Thus, by Proposition \ref{p4}, one gets $[\myphi_{G'}]\equiv [\myphi^{1,2}_{G',3},\myphi^{13,23}_G]\mod q^3$. We are again able to apply Chevalley-Warning (Theorem \ref{T18}) for the two polynomials $\myphi^{1,2}_{G',3},\myphi^{13,23}_G$ and get $[\myphi_{G'}]\equiv 0\mod q^3$. 
\end{proof}

\begin{corollary}\label{cor34}
Let $G$ be a planar graph. Then $G$ is duality admissible.
\end{corollary}
\begin{proof}
If $G$ is planar, then each subquotient graph $G\backslash I \q J$ are also planar. The conditions in Proposition \ref{prop33} are satisfied, thus $G$ is duality admissible.
\end{proof}
\noindent In general, the essential part of the conditions in subquotient graphs in Proposition ref{prop33} is the existence of a 3-loop in any subquotient graph, that allows us to get good divisibility conditions for $[\myphi^J_I]$ by Proposition \ref{p4}. The corresponding divisibility for the dual situation, i.e. for $[\Psi^I_J]$, is "easier" to be satisfied since a log-divergent graph always has a 3-valent vertex. An example of a log-divergent graph that has no 3-loops can be found in \cite{Sch} on Figure 1,d)  (after deletion of one of the vertices). We can also extend the ideas to the graphs that possibly have no triangles, but have a 4-loop. This was done in \cite{D3}. The graphs without 4-loops (i.e. graphs of girth $\geq 5$) are too big and special for being interested from the physical point of view.
\medskip

An interesting set of subquotient graphs of $G$ not covered by the conditions on $I$ and $J$ in Definition \ref{def26} and Proposition \ref{lemma8} is formed by the graphs $\gamma$ with $h_{\gamma},n_{\gamma}\leq 2$. We refer to such graphs as \emph{small} graphs. For a graph $G$, denote by
\begin{equation}\label{f15}
R^{u,v}(G):= \big\{ \gamma=G \backslash I\q J \;\big|\; \gamma\text{ is conn. and co-conn.},\; |I|=h_G-u, |J|=n_G-v \big\}, 
\end{equation}
and $r^{u,v}(G):= | R^{u,v}(G)|$. By co-connected we mean that no self-loop has been contracted. We also define $\overline{R}^{u,v}(G)\supset R^{u,v}(G)$ for the same set but without condition "connected and co-connected", and $\bar{r}^{u,v}:=|\overline{R}^{u,v}(G)|$.
One can easily compute 
\begin{equation}\label{f16}
\bar{r}^{u,v}(G)=\frac{N_G!}{u!v!(N_G-u-v)!}=\bar{r}^{v,u}(G).
\end{equation}
The numbers $r^{u,v}(G)$ are well-understood in the case $u=0$ or $v=0$.
\begin{proposition}
Let $G$ be a graph. Then, for $u\geq 0$,
\begin{equation}
r^{u,0}(G)=\binom{h_G}{u}\cdot\#\{\text{spanning trees of G}\}
\end{equation}
and 
\begin{equation}\label{prop34}
r^{0,u}(G)=\binom{n_G}{u}\cdot\#\{\text{spanning trees of G}\}.
\end{equation}
\end{proposition}
\begin{proof}
Let $\gamma=G \backslash I\q J$ be a subquotient graph such that it is connected and co-connected with $h_{\gamma}=u$ and $n_{\gamma}=0$. To obtain $\gamma$ from $G$, we can first contract $n_G$ edges in $J$. Since $\gamma$ is co-connected, these edges form no cycles, so they build a spanning tree. We obtain a dot with $h_G$ self-loops and we need to delete $|I|$ of them. So we get the binomial coefficient.

Similarly for the second part: we first delete $h_G$ edges and see that $\gamma$ is connected if these edges form a complement of a spanning tree.  
\end{proof}
 
\begin{corollary}\label{cor35}
Let $G$ be a log-divergent graph. Then for any $u$, $0\leq u\leq h_G$,
\begin{equation}
r^{u,0}(G)=r^{0,u}(G).
\end{equation}
\end{corollary}
\begin{proof}
The statement trivially follows from Proposition \ref{prop34} since $h_G=n_G$ for a log-divergent graph by definition.
\end{proof}
  
The numbers $r^{u,v}$ of small subquotient graphs for different $u$ and $v$ are a part of the local information about $G$ and are hard to control for $u\neq 0 \neq v$. Non the less, the numbers $r^{1,2}$ and $r^{2,1}$ will appear in the proof of the main theorem. Here we analyse one important relevant example.

\noindent
\begin{picture}(200,100)(-20,10)
\thicklines
\qbezier(20,60)(30,90)(40,60) 
\qbezier(20,60)(30,30)(40,60)
\qbezier(20,60)(30,60)(40,60)
\qbezier(40,60)(50,75)(60,60)
\qbezier(40,60)(50,45)(60,60)
\qbezier(60,60)(70,75)(80,60)
\qbezier(60,60)(70,45)(80,60)
\put(85,60){\circle*{2}}
\put(90,60){\circle*{2}}
\put(95,60){\circle*{2}}
\qbezier(100,60)(110,75)(120,60)
\qbezier(100,60)(110,45)(120,60)
\qbezier(120,60)(130,60)(140,60)
\qbezier(300,60)(287,65)(287,60)
\qbezier(300,60)(287,55)(287,60)
\qbezier(300,60)(286,72)(276,60)
\qbezier(300,60)(286,48)(276,60)
\put(270,60){\circle*{2}}
\put(265,60){\circle*{2}}
\put(260,60){\circle*{2}}
\qbezier(300,60)(278,85)(251,60)
\qbezier(300,60)(278,35)(251,60)
\qbezier(300,60)(278,95)(230,60)
\qbezier(300,60)(278,25)(230,60)
\qbezier(300,60)(282,95)(210,80)
\qbezier(300,60)(282,25)(210,40)
\qbezier(210,80)(213,60)(210,40)

\put(20,60){\circle*{3}}
\put(40,60){\circle*{3}}
\put(60,60){\circle*{3}}
\put(80,60){\circle*{3}}
\put(100,60){\circle*{3}}
\put(120,60){\circle*{3}}
\put(140,60){\circle*{3}}

\put(300,60){\circle*{3}}
\put(276,60){\circle*{3}}
\put(251,60){\circle*{3}}
\put(230,60){\circle*{3}}
\put(210,80){\circle*{3}}
\put(210,40){\circle*{3}}
\thinlines
\put(125,20){$G_n$}
\put(290,20){$G^{dual}_n$}
\end{picture}
\begin{center}
{\textsc{Figure 3}}.
\end{center}
\begin{lemma}\label{lem36}
For a given $n\geq 2$, let $G_n$ be the log-divergent graph with $N_G=2n$ edges of the shape (left to right) : triple edge, n-2 copies of a double edge, single edge (see Figure 3 above). Let $G^{dual}_n$ be the planar dual to $G_n$. Then 
\begin{equation}\label{f18}
r^{1,2}(G_n)=3\cdot 2^{n-3}n(n-1)^2
\end{equation}
and
\begin{equation}\label{f19}
r^{2,1}(G_n)=r^{1,2}(G^{dual}_n)= 2^{n-2} + 3\cdot 2^{n-3}\cdot(n-1)n^2.
\end{equation}
\end{lemma}
\begin{proof}
First we prove (\ref{f18}). The set of small subquotient graphs $\gamma=G\backslash I\q J$ for $r^{1,2}(G_n)$ in (\ref{f15}) can be represented by $A\cup B\cup C\cup D$, where each $\gamma$ in $A$ and $B$ (resp. $C$ and $D$) was obtained by deletion of one (resp. two) of the edges of a triple edge in $G$, and each $\gamma$ in  $A$ and $C$ (resp. $B$ and $D$) has a double edge (resp. self-loop). Then one computes :
\begin{multline}
r^{1,2}(G_n)=|A|+|B|+|C|+|D|=3\cdot 2^{n-2} (n-1)
+ 3\cdot 2^{n-2}(n-1)(n-2)\\
+3 (n-2)\cdot2^{n-3}(n-1) + 3(n-2)\cdot 2^{n-3}(n-1)(n-2)\\
=3\cdot 2^{n-3}\cdot n(n-1)^2.
\end{multline}
Now we prove (\ref{f19}). The corresponding set of subquotient graphs $\gamma$ for $r^{1,2}(G^{dual}_n)$ can be represented by $A\cup B\cup C$, where for $\gamma$ in $B\cup C$ (resp. $A$) the initial self-loop of $G_n$ was (resp. was not) deleted, and in $B$ (resp. $C$) an edge of the triangle was (was not) deleted. Analysing separately, one gets
\begin{multline}
r^{1,2}(G^{dual})=|A|+|B|+|C|=3\cdot 2^{n-2}\cdot n(n-1) \\+ 3\cdot 2^{n-3}(n-2)\cdot\big((n-1)+(n-1)(n-2)\big) \\
+ 2^{n-2}\big(1+3(n-2)+3(n-2)(n-3)/2\big)\\
= 2^{n-2} + 3\cdot 2^{n-3}\cdot(n-1)n^2.
\end{multline} 
\end{proof}

\begin{remark} In contrast to the equality $r^{0,u}(G)=r^{0,u}(G^{dual})$, in the example above we see, that for each even $N_G=2n$, the number $r^{1,2}(G)$ is not necessarily stable under duality. Indeed, the two computed values for the graph $G_n$ on Picture 1 are different, for each $n\geq 2$. 
\end{remark}
\medskip
\noindent All the preparations are done and we are ready to prove the main theorem of this section.
\begin{theorem}\label{Thm2}
Let $G$ be a duality admissible graph with $h_G,n_G\geq 2$. Then 
\begin{equation}
c_2(G)_q=c_2^{dual}(G)_q.
\end{equation}
\begin{proof}
Define $n:=n_G=h_G$, $N:=N_G=2n$. Let $\Psi=\Psi_G$ be the graph polynomial and $\myphi=\myphi_G$ the dual one.
Denote by $\cP$ the $\QQ$-algebra generated by the sums of the point-counting functions. It is spanned by the functions $q\mapsto \#Y(\FF_q)$ from the set of prime powers to integers with $Y\in Var_{\QQ}$. Consider the elements $S_t:=\sum_{I,J}[\Psi^I_J]'$, where the sum goes over all $I,J\subset I_N$ with $|I|=|J|=t$, $t=1,\ldots,n$. Identity (\ref{c80}) shows that $S_t$ respects Cremona transformation, i.e. symmetric under $(\Psi\leftrightarrow \myphi)$. By (\ref{sum2}), $S_t$ is in $\cP$ for any $t$. One also has $q^3:=[\AAA^3]\in \cP$. 

Let $\cI\subset \cP$ be the ideal generated by $q^3$ and all $S_t$, $1\leq t\leq n-1$.

\noindent We start with $\Psi$ and apply formula (\ref{sum1}):
\begin{equation}\label{c132}
[\Psi]=[\Psi]'+\sum_{t=1}^{N}\sum_{|I|=t} [\Psi_{I}]'. 
\end{equation}
Using the duality $[\Psi^J_I]'=[\myphi^I_J]'$ for all $I,J\subset I_N$, one gets 
\begin{equation}
[\Psi]=[\myphi]'+\sum_{t=1}^{N}\sum_{|I|=t} [\myphi^{I}]'. 
\end{equation}
We always assume $\Psi^J_I=0$ and $\myphi^I_J=0$ for $I\cap J \neq\emptyset$.
For each $[\myphi^I]'$ we substitute the expression from (\ref{sum2}) applied to $\Psi:=\myphi^I$ and get 
\begin{equation}\label{c133}
[\Psi]=[\myphi]'+\sum_{t=1}^{N}\sum_{|I|=t} \Big([\myphi^{I}]+\sum_{s=1}^{N-t}(-1)^{s}\sum_{|J|=s} [\myphi^{I}_J]\Big). 
\end{equation}
We know that $[\myphi^I_J]=[\myphi_{G'}]=0=[\Psi^J_I]$ with $G'=G\backslash J\q I$ for $|I|>n$ or $|J|>n$, so we can reduce the upper bound of the summation signs from $N$ to $n$. Since $[\myphi^I_J]\equiv 0 \mod q^3$ by Proposition \ref{prop33}, for all $I,J \subset I_N$ with $|I|>|J|$ and $|J|\leq n-3$, we forget these summands shifting to the computations modulo $q^3$. There are also summands $[\myphi^I_J]$ with $|I|>|J|\geq n-2$. In other words, these are the summands $[\myphi_{\gamma}]$ for small graphs $\gamma=G\backslash J\q I$ with $n_{G'}< h_{G'}\leq 2$. 
We will collect all terms $[\myphi_{\gamma}]$ we get for such small graphs  (together with the dual objects of the next steps) to the sum denoted by $A_1$ (respectively $A_r$ on the $r$-th step). 
 
Now, the summands $[\myphi^I_J]$ of the last brackets of (\ref{c133}) with $|I|=|J|=t$ do not need to be 0, but they sum up to the element $S_t\in \cI$. Thus one gets 
\begin{equation}\label{c135}
[\Psi]\equiv [\myphi]'+\sum_{t=1}^{n}\sum_{|I|=t} \sum_{s=t+1}^{n-t}(-1)^{s}\sum_{|J|=s}[\myphi^{I}_J]+ A_1\mod \cI.
\end{equation}
Using induction on $r$, $1\leq r\leq N$, we now prove the following statement: 
\begin{equation}\label{c136}
[\Psi]\equiv \left\{
\begin{aligned}\mathstrut [\myphi]'+\sum_{t=r}^{n}\sum_{|I|=t} \sum_{s=t+1}^{n-t}d^{(r)}_{t,s}\sum_{|J|=s}[\myphi^{I}_J] +A_r \mod \cI, \;r \text{\;odd,\;}\\
\mathstrut[\Psi]'+\sum_{t=r}^{n}\sum_{|I|=t} \sum_{s=t+1}^{n-t}d^{(r)}_{t,s}\sum_{|J|=s}[\Psi^{I}_J] + A_r \mod \cI,\; r\text{\;even}.
\end{aligned}\right.
\end{equation}
Here $A_r$ is again a sum of terms $[\myphi_{G'}]$ for small graphs and the duals $[\Psi_{G'}]$.
Formula (\ref{c135}) is the base of the induction, $r=1$ and $d_{t,s}^{(1)}=(-1)^s$.
For general $r$, we first start with an odd $r$ and the congruence 
\begin{equation}\label{c138}
[\Psi]\equiv [\myphi]'+\sum_{t=r}^{n}\sum_{|I|=t} \sum_{s=t}^{n-t}d^{(r)}_{t,s}\sum_{|J|=s}[\myphi^{I}_J]+A_r \mod \cI.
\end{equation}
The application of (\ref{sum2}) for each $\Psi:=\myphi^I_J$ yields: 
\begin{equation}
[\Psi]\equiv [\myphi]'+\sum_{t=r}^{n}\sum_{|I|=t} \sum_{s=t+1}^{n-t}d_{t,s}^{(r)}\sum_{|J|=s}\Big([\myphi^{I}_J]'+ \sum_{p=1}^{n-t-s}\sum_{|K|=p}[\myphi^I_{JK}]' \Big)+A_r\!\mod \cI
\end{equation} 
with the rightmost summation going over all $K\subset I_N\backslash (I+J)$. Collecting the summands by the  cardinality of indexes, we get 
\begin{equation}
[\Psi]\equiv [\myphi]'+\sum_{t=r}^{n}\sum_{|I|=t} \sum_{s=t+1}^{n-t}b_{t,s}^{(r)}\sum_{|J|=s}[\myphi^{I}_J]' +A_r\mod \cI.
\end{equation} 
The coefficients $b^{(r)}_{t,s}$ depend only on $d^{(r)}_{i,j}$, $i=|I|\leq t$, $j=|J|\leq s$, but not on $I$ and $J$ itself. Using the duality, we rewrite
\begin{equation}\label{c141}
[\Psi]\equiv [\Psi]'+\sum_{t=r}^{n}\sum_{|I|=t} \sum_{s=t+1}^{n-t}b_{t,s}^{(r)}\sum_{|J|=s}[\Psi^{J}_I]' + A_r \mod \cI.
\end{equation} 
Now, using (\ref{sum1}) for each $\Psi=\Psi^J_I$, one can rewrite the formula above as 
\begin{multline}\label{c142}
[\Psi]\equiv [\Psi]'+\sum_{t=r}^{n}\sum_{|I|=t} \sum_{s=t+1}^{n-t}b_{t,s}^{(r)}\sum_{|J|=s}\Big([\Psi^{J}_I]+ (-1)^p\sum_{p=1}^{n-t-s}\sum_{|K|=p}[\Psi^J_{IK}] \Big) \\+ A_r\mod \cI.
\end{multline}
By Proposition \ref{lemma8}, we can get rid of all the summands $[\Psi^{J'}_{I'}]$ for $|J'|\geq |I'|$, $|I'|\leq n-3$, while the sums $\sum_{I',J'}[\Psi^{J'}_{I'}]$, $|I'|=|J'|$ contribute to $0\mod \cI$. We also sum up all the terms for small graphs (here $\gamma$ with $h_{\gamma}<n_{\gamma}\leq 2$); adding $A_r$, we denote the result by $A_{r+1}$.

\noindent Collecting the remaining summands by the cardinality of indexes, one gets 
\begin{equation}\label{c143}
[\Psi]\equiv [\Psi]'+\sum_{s=r+1}^{n}\sum_{|J|=s}\sum_{t=s+1}^{n-s}\sum_{|I|=t} d_{s,t}^{(r+1)} [\Psi^J_{I}] + A_{r+1}\mod \cI
\end{equation}
for some integer coefficients $d_{s,t}^{(r+1)}$ (linearly) depending on $b^{(r)}_{i,j}$, $i\leq t$, $j\leq s$.
\noindent
\begin{picture}(120,125)
\thicklines
\put(20,20){\vector(1,0){90}}
\put(20,20){\vector(0,1){90}}
\put(150,20){\vector(1,0){90}}
\put(150,20){\vector(0,1){90}}
\put(280,20){\vector(1,0){90}}
\put(280,20){\vector(0,1){90}}
\thinlines
\put(40,40){\circle*{2}} \put(60,60){\circle*{2}}
\put(40,60){\circle*{3}} \put(60,40){\circle*{2}}
\put(60,80){\circle*{3}} \put(80,60){\circle*{2}} \put(80,80){\circle*{2}}
\put(40,80){\circle*{3}} \put(80,40){\circle*{2}} 
\put(40,100){\circle*{3}} \put(60,100){\circle*{3}} \put(80,100){\circle*{3}}  \put(100,40){\circle*{2}} \put(100,60){\circle*{2}} \put(100,80){\circle*{2}}\put(100,100){\circle*{2}} 
\put(36,51){\line(0,1){56}} 
\put(36,51){\line(1,1){54}}
\put(38,14){$\mathstrut_r$}\put(54,14){$\mathstrut_{r+1}$}
\put(10,40){$\mathstrut_r$} \put(4,58){$\mathstrut_{r+1}$} \put(12,12){$\mathstrut_0$} \put(11,100){$J$} \put(95,10){$I$}

\put(170,40){\circle*{2}} \put(190,60){\circle*{3}}
\put(170,60){\circle*{3}} \put(190,40){\circle*{2}}
\put(190,80){\circle*{3}} \put(210,60){\circle*{3}} \put(210,80){\circle*{3}}
\put(170,80){\circle*{3}} \put(210,40){\circle*{2}} 
\put(170,100){\circle*{3}} \put(190,100){\circle*{3}} \put(210,100){\circle*{3}}  \put(230,40){\circle*{2}} \put(230,60){\circle*{3}} \put(230,80){\circle*{3}}\put(230,100){\circle*{3}} 
\put(166,56){\line(0,1){51}} 
\put(166,56){\line(1,0){70}}
\put(168,14){$\mathstrut_r$}\put(184,14){$\mathstrut_{r+1}$}
\put(140,40){$\mathstrut_r$} \put(134,58){$\mathstrut_{r+1}$} \put(142,12){$\mathstrut_0$} \put(141,100){$J$} \put(220,10){$IK$}

\put(300,40){\circle*{2}} \put(320,60){\circle*{2}}
\put(300,60){\circle*{2}} \put(320,40){\circle*{2}}
\put(320,80){\circle*{2}} \put(340,60){\circle*{3}} \put(340,80){\circle*{2}}
\put(300,80){\circle*{2}} \put(340,40){\circle*{2}} 
\put(300,100){\circle*{2}} \put(320,100){\circle*{2}} \put(340,100){\circle*{2}}  \put(360,40){\circle*{2}} \put(360,60){\circle*{3}} \put(360,80){\circle*{3}}\put(360,100){\circle*{2}} 
\put(329,56){\line(1,1){36}} 
\put(329,56){\line(1,0){38}}
\put(298,14){$\mathstrut_r$}\put(314,14){$\mathstrut_{r+1}$}
\put(270,40){$\mathstrut_r$} \put(264,58){$\mathstrut_{r+1}$} \put(272,12){$\mathstrut_0$} \put(334,14){$\mathstrut_{r+2}$}
 \put(271,100){$J$} \put(355,10){$I$}
\end{picture}
\begin{center}
\textsc{Figure 4.}   
\end{center}\medskip
On the Figure 4 on the left there are indicated the pairs $(I,J)$ for which the summands $[\Psi^J_I]'$ appear in formula (\ref{c141}). The middle picture shows the pairs $(J,I)$ and $(J,IK)$ such that $\Psi^J_{IK}$ appear in formula (\ref{c142}). The right picture shows what summands $[\Psi^J_I]$ survive in (\ref{c143}). Reflecting the right picture, we see that we have decreased the number of the (fat) points (terms surviving in the sum) by 1 level.

So, interchanging $s$ and $t$, as well as $I$ and $J$ in (\ref{c143}), one obtains the statement for $r+1$ in (\ref{c136}):
\begin{equation}\label{c144}
[\Psi]\equiv [\Psi]'+\sum_{t=r+1}^{n}\sum_{|I|=t}\sum_{s=t+1}^{n-t}\sum_{|J|=s} d_{s,t}^{(r+1)} [\Psi^I_{J}] + A_{r+1}\mod \cI.
\end{equation}
The conditions (duality and vanishing lemmas) we used above are symmetric under $\Psi\leftrightarrow \myphi$ in the right hand side of equations (\ref{c138}) - (\ref{c143}). This implies the proof for the case $r$ is even starting with formula (\ref{c138}) after substituting $\myphi=\Psi$. This finishes our inductive proof of (\ref{c136}).

The polynomials $\Psi$ and $\myphi$ are of degree $n$  of $N=2n$ variables. On the $r=(n-3)$-rd step we get rid of all the summands in the big sums on the right of (\ref{c136}). Indeed, consider the case $r$ is odd. On that step we derive (\ref{c143}) with terms with $|J|>|I|\geq n-2$ (corresponding to small graphs). But these terms are considered to be in $A_{r+1}$ already. The same holds in the case $r$ is even. 

So we get $[\Psi]\equiv [\myphi]'+A_{n-2}\equiv [\Psi]'+A_{n-2}\mod \cI$. 
In other words,
\begin{equation}\label{c145}
[\Psi]= [\Psi]' + a(\Psi)+ \sum_{i=1}^n u_i(\Psi)S_i+v(\Psi)q^3
\end{equation}
with $a(\Psi):=A_{r-2},v(\Psi),u_i(\Psi)\in \cP$, $1\leq i\leq N$.

Now we want to do the similar computation starting with $[\myphi]$ in the left hand side of (\ref{c132}). One can again use the symmetry between $\Psi$ and $\myphi$ in the applied conditions (Proposition \ref{lemma8}) and duality. Starting with formula (\ref{c132}), we do the same swapping $\Psi$ with $\myphi$ both on the left and on the right hand side of each formula until we finally derive
\begin{equation}\label{c146}
[\myphi]= [\myphi]' + a(\myphi)+ \sum_{i=1}^n u_i(\myphi)S_i+v(\myphi)q^3
\end{equation}
for $a(\myphi),v(\myphi),u_i(\myphi)\in \cP$, $1\leq i\leq n$. We do not have control on the relation between $v(\Psi)$ and $v(\myphi)$, but the coincidence of the coefficients $d_{t,s}^{(r)}$, $b_{t,s}^{(r)}$ in (\ref{c132})-(\ref{c145}) for $\Psi$ and $\myphi$ yields  
$u_i(\Psi)=u_i(\myphi)$ for each $i$, $1\leq i\leq n$. Now (\ref{c145}) and (\ref{c146}) imply 
\begin{equation}\label{c147}
[\Psi]-[\myphi]=
 (v(\Psi)-v(\myphi))q^3 + (a(\Psi)-a(\myphi)).
\end{equation}
Let's show that $a(\Psi)=a(\myphi)$, i.e. $a$ is stable under duality. 

By the discussion before (\ref{c143}) for $\Psi$ in the case $r$ is odd, $A_{r+1}(\Psi)\backslash A_r(\Psi)$ is a sum of terms of the form $[\Psi_{\gamma}]$ for small graphs $\gamma=G\backslash J\q I$ with $h_{\gamma}<n_{\gamma}\leq 2$, multiplied by some coefficients. If $\gamma\in R^{0,1}$ (see (\ref{f15})), there is only unique such subquotient graph up to isomorphism. Then $\Psi_{\gamma}=1$ and $[\Psi_{\gamma}]=0$. Otherwise, if $\gamma$ with $h_{\gamma}=0$ and $n_{\gamma}=1$ is disconnected or dis-co-connected, then $\Psi_{\gamma}=0$ and $[\Psi_{\gamma}]=q$. There is also a unique subgraph $\gamma\in R^{0,2}$ up to isomorphism, it gives $\Psi_{\gamma}=1$ and $[\Psi_{\gamma}]=0$, while in a disconnected or dis-co-connected situation we get $[\Psi_{\gamma}]=q^2$. The the last case of a small graph, for $\gamma\in R^{1,2}$, there are 4 different possible non-isomorphic graphs, but they all give the same $[\Psi_{\gamma}]=q^2$, while in a disconnected or dis-co-connected situation we get $[\Psi_{\gamma}]=q^3$. 

Similar to $d^{(r+1)}_{s,t}$ in (\ref{c143}), the coefficients of small graph $\gamma\in R^{u,v}$ depend only on the values $u$ and $v$, but not on the edges we delete and contract, and the dependence is linear on the coefficients of the previous step, so we get some  expressions of binomial coefficients, denote them $\tilde{d}^{(r+1)}_{u,v}:=d^{(r+1)}_{n-u,n-v}$. Since we know the number of connected and co-connected small graphs for $u=0$ or $v=0$ by Corollary \ref{cor35}, we can compute 
\begin{multline}\label{c150}
A_{r+1}(\Psi)\backslash A_r(\Psi)= \tilde{d}^{(r+1)}_{0,1}(\Psi)\big(\bar{r}^{0,1}(\Psi)-r^{0,1}(\Psi)\big)q + 
\\\tilde{d}^{(r+1)}_{0,2}(\Psi)\big(\bar{r}^{0,2}(\Psi)-r^{0,2}(\Psi)\big)q^2 +\\
 \tilde{d}^{(r+1)}_{1,2}(\Psi)\big(r^{1,2}(\Psi)q^2+(\bar{r}^{1,2}(\Psi)-r^{1,2}(\Psi))q^3\big),
\end{multline}
where $r^{u,v}(\Psi):=r^{u,v}(G)$ and $\bar{r}^{u,v}(\Psi):=\bar{r}^{u,v}(G)$ are the numbers from (\ref{f15}).

Now suppose we start with $[\myphi]$ and use the same reduction procedure as in (\ref{c136}). We again collect the sums of small graphs into $A_r's$. When we restrict our attention to the case $r$ is odd, do the same as above, and we get the expression for $A_{r+1}(\myphi)\backslash A_r(\myphi)$ similar to (\ref{c150}).  Analysing the small classes in $A_{r+1}(\myphi)\backslash A_r(\myphi)$, one obtains same values $[\myphi_{\gamma}]=q^i$ or 0. We know also that $\tilde{d}^{(r+1)}_{u,v}(\Psi)=\tilde{d}^{(r+1)}_{u,v}(\myphi)$, since it depends only on the number of steps and on the number $u$ and $v$, but not on $I$ and $J$ itself (not on the local topology of the graph). We also know that $r^{0,v}(\Psi):=r^{0,v}(G)=r^{v,0}(G)=:r^{0,v}(\myphi)$ for $v=1$ and 2, and also $r^{u,v}(\Psi)=r^{u,v}(\myphi)$ by (\ref{f16}). Comparing the two equations of the form (\ref{c150}) for $\Psi$ and $\myphi$, we derive 
\begin{multline}
A_{r+1}(\Psi)\backslash A_r(\Psi) - A_{r+1}(\myphi)\backslash A_r(\myphi) =
\tilde{d}^{(r+1)}_{1,2}\big(r^{1,2}(\Psi)q^2+\\ (\bar{r}^{1,2}(\Psi)-r^{1,2}(\Psi))q^3\big)- \tilde{d}^{(r+1)}_{1,2}\big(r^{1,2}(\myphi)q^2+(\bar{r}^{1,2}(\myphi)-r^{1,2}(\myphi))q^3\big)=\\
d^{(r+1)}_{1,2}q^2(1-q)\big(r^{1,2}(G)-r^{2,1}(G)\big).
\end{multline}

Let us look at the situation for $[\Psi]$ again but on the even step $r$ of reduction. Then the classes $\Psi_{\gamma}$ becomes $\myphi_{\gamma}$ on the right hand side of (\ref{c136}) and the situation is similar to the odd step for $[\myphi]$, and vice-versa. So one obtains 
\begin{equation}
(A_{r+1}\backslash A_r)(\Psi) - (A_{r+1}\backslash A_r)(\myphi) = d^{(r+1)}_{2,1}q^2(1-q)\big(r^{2,1}(G)-r^{1,2}(G)\big)
\end{equation}
with $r$ even.

We can sum over all $r$ and obtain an equality in terms of $a=\sum_r A_{r+1}\backslash A_r$ :
\begin{equation}
a(\Psi)-a(\Phi) = Cq^2(1-q)(r^{1,2}(G)-r^{2,1}(G))
\end{equation} 
with a particular coefficient $C$. This coefficient depends only on the number of steps (the size of $G$) but not on $G$ itself. So $C=C(n)$, where $N_G=2n$ is the number of edges of our log-divergent graph $G$.

We return to (\ref{c147}) and write 
\begin{equation}\label{f39}
[\Psi]-[\myphi]=
 (v(\Psi)-v(\myphi))q^3 + C(n)q^2(1-q)(r^{1,2}(G)-r^{2,1}(G)).
\end{equation}

Our next step is to show that the coefficient $C(n)$ is 0 for each $n$. Consider again the example $G=G_n$ from Figure 3 for a fixed $n\geq 2$. Direct simple computation yields 
\begin{equation}
\begin{aligned}
\Psi_{G_n}=(\alpha_1\alpha_2+\alpha_2\alpha_3+\alpha_1\alpha_3)(\alpha_4+\alpha_5)\ldots (\alpha_{2n-2}+\alpha_{2n-1}),\\
\myphi_{G_n}=(\alpha_1+\alpha_2+\alpha_3)(\alpha_4+\alpha_5)\ldots (\alpha_{2n-2}+\alpha_{2n-1})\alpha_{2n}
\end{aligned}
\end{equation}
for obvious numeration of edges from left to right on the figure.
Thus, on the level of point counting, $q^3|[\Psi]$ and $q^3|[\myphi]$. Equation (\ref{f39}) for $G=G_n$ now implies 
\begin{equation}
C(n)q^2(1-q)(r^{1,2}(G_n)-r^{2,1}(G_n))\equiv 0 \mod q^3.
\end{equation}
By Lemma \ref{lem36}, we know that $r^{1,2}(G_n)-r^{2,1}(G_n)=F(n)$ is the polynomial expression of $n$ and $2^n$. The congruence above implies that $C(n)$ is divisible by $q$, for every prime power $q\not| F(n)$. That is why $C(n)=0$. 

Since $C(n)$ vanishes, we derive from (\ref{f39}) that
\begin{equation}\label{f40}
[\Psi_G]-[\myphi_G]\equiv 0 \mod q^3
\end{equation}
for any log-divergent graph $G$ with $N_G=2n$ edges, for any given $n\geq 3$.
Since $[\Psi_G]\equiv q^2\cdot c_2(G)_q \mod q^3$ and
$[\myphi_G]\equiv q^2\cdot c_2^{dual}(G)_q \mod q^3$, Formula $(\ref{f40})$ finally yields
\begin{equation}
c_2^{dual}(G)_q=c_2(G)_q.
\end{equation}
\end{proof}
\end{theorem}

\end{document}